\documentclass[12pt]{article}
\usepackage{amsmath}
\usepackage{amsfonts}
\usepackage{amssymb}
\usepackage{latexsym}

\newcommand{\ulxi} {{\underline \xi}}
\newcommand{\ulx} {{\underline x}}
\newcommand{\uly} {{\underline y}}
\newcommand{\ulw} {{\underline w}}

\newcommand{\bfe} {{\mathbf e}}
\newcommand{\ulzeta} {{\underline \zeta}}
\newcommand{\esssup}{{\mathrm{ess}\sup}}

\newtheorem{thm}{\bf Theorem}[section]
\newtheorem{lem}[thm]{\bf Lemma}
\newtheorem{prop}[thm]{\bf Proposition}

\newenvironment{proof}{\noindent{\em Proof:}}{\quad \hfill$\Box$\vspace{2ex}}

\newenvironment{remark}{\noindent{\bf Remark}}{\vspace{2ex}}

\AtEndDocument{\bigskip{\footnotesize
		\textsc{Faculty of Information Technology, Macau University of Science and Technology, Macao, China}\par
		\textit{E-mail address: }\texttt{pdang@must.edu.mo}\par
		\medskip
		\textsc{Macao Center for Mathematical Sciences, Macau University of Science and Technology, Macao, China}\par
		\textit{E-mail address: }\texttt{wxmai@must.edu.mo}\par
		\medskip
		\textsc{Macao Center for Mathematical Sciences, Macau University of Science and Technology, Macao, China}\par
		\textit{E-mail address: }\texttt{tqian@must.edu.mo}
}}

\title{On Monogenic Reproducing Kernel Hilbert Spaces of the Paley-Wiener Type \footnotetext{This work was supported in part by the Science and Technology Development Fund, Macau SAR: 154/2017/A3; NSFC Grant No. 11701597; NSFC Grant No. 11901594; The Science and Technology Development Fund, Macau SAR: 079/2016/A2, 0123/2018/A3.}}

\author{Pei Dang, Weixiong Mai\thanks{Corresponding author}, Tao Qian}

\date{}

\begin{document}
 \maketitle
 \begin{center}
 	\begin{minipage}{120mm}
 		\begin{center}{\bf Abstract}\end{center}
 		{
 			In the Clifford algebra setting the present study develops three reproducing kernel Hilbert spaces of the Paley-Wiener type, namely the Paley-Wiener spaces, the Hardy spaces on strips, and the Bergman spaces on strips. In particular, we give spectrum characterizations and representation formulas of the functions in those spaces and estimation of their respective reproducing kernels.
 			
 		}
 	\end{minipage}
 \end{center}
 \tableofcontents
 \section{Introduction}
In this paper we will study three reproducing kernel Hilbert spaces (RKHS) in the Clifford algebra setting. They are the $PW(\pi,\mathbf C^{(m)})$ Paley-Wiener space, the $H^2(S_a,\mathbf C^{(m)})$ Hardy space on a strip $S_a,$  and the $A^2(S_a,\mathbf C^{(m)})$ Bergman space on a strip $S_a$, where $S_a=\{x=x_0+\ulx\in \mathbf R^{m+1}:\ulx\in \mathbf R^m, |x_0|<a\}\subset \mathbf R^{m+1}.$ The three spaces are closely related to the well-known classical Paley-Wiener theorems referred to the Hardy $H^2$ space in the upper-half complex plane and entire functions with certain exponential increasing at the infinity on the whole complex plane (\cite{Paley-Wiener}).
The upper-half complex plane Hardy space version is stated as follows: $f\in L^2(\mathbf R)$ is the nontangential boundary limit (NTBL) function of some function in the Hardy $H^2$ space of the upper-half plane (denoted by $H^2(\mathbf C_+)$) if and only if $\hat f= \chi_{[0,\infty)}\hat f,$ where $\hat f$ is the Fourier transform of $f,$ which is phrased as the non-compact type Paley-Wiener theorem in this paper. The entire function version is that $f\in L^2(\mathbf R)$ is the restriction of an entire function $f(z)$ with the bounds $C\exp (\pi |z|)$ if and only if $\hat f=\chi_{[-\pi,\pi]}\hat f.$ This will be phrased as the compact type Paley-Wiener theorem in this paper.

There exist analogous results of the Paley-Wiener theorems in higher dimensions, that are formulated with, respectively, the several complex variables and the Clifford algebra settings.

In the several complex variables setting the Paley-Wiener theorem is for the Hardy spaces on tubes over regular cones, $H^2(T_\Gamma),$ where $\Gamma\subset\mathbf R^m$ is any regular cone and $T_\Gamma=\{x+iy\in \mathbf C^m; x\in \mathbf R^m,y\in \Gamma\}$ (see \cite{SW}), as a generalization of the non-compact case.  The Paley-Wiener Theorem states that $f\in H^2(T_\Gamma)$ if and only if $\hat f=\chi_{\Gamma^*}\hat f,$ where $\Gamma^*$ is the dual cone of $\Gamma.$ We cite also analogous results for Bergman spaces on tubes over regular cones (see e.g. \cite{BBGNPR,BG,Genchev}).
As an analogue of the compact case, the Paley-Wiener theorem is generalized to entire functions of several complex variables with the exponential type bounds (see e.g. \cite{SW}). This type of holomorphic functions corresponding to those whose Fourier transforms $\mathbf R^m$ are supported in compact convex sets.

In the Clifford algebra setting a compact type Paley-Wiener theorem is obtained in \cite{Kou-Qian}. A standard non-compact version is as follows. Denote by $H^2(\mathbf R^{m+1}_+,\mathbf C^{(m)})$ the $\mathbf C^{(m)}$-valued Hardy space on the upper-half space, $\mathbf R^{m+1}_+=\{x=x_0+\ulx\in \mathbf R^{m+1}_+: x_0>0,\ulx\in \mathbf R^m \}$. Then $f\in H^2(\mathbf R^{m+1}_+,\mathbf C^{(m)})$ if and only if the nontangential boundary limit $f$ satisfies $\hat f=\chi_+\hat f,$ where $\chi_+(\ulxi)= \frac{1}{2}(1+i\frac{\ulxi}{|\ulxi|}).$ Moreover, the last relation holds if and only if $f=\frac{1}{2}(I+H)f,$ where $H=-\sum_{j=1}^m R_j\bfe_j$ and $R_j$'s are the Riesz transformations. This result is an alternative version of the result on the conjugate harmonic systems \cite{SW,Stein2}. So far the Paley-Wiener type theorems have been extensively studied that include generalizations in the distribution sense to the $L^p$ cases, $1\leq p\leq \infty$, as well as analogues in the Bergman and Dirichlet spaces (see e.g. \cite{Bernstein,Q1,QXYYY,Li-Deng-Qian,Dang-Mai-Qian,Schwartz,Hor,Gilbert-Murray,Durenetal,Garrigos}).

The aim of the present paper is two-fold. One is to obtain the Fourier transform characterizations of the above mentioned Clifford monogenic spaces; and the other is to show that they are reproducing kernel Hilbert spaces (RKHSs). Their reproducing kernels are computed and estimated. 

Denote by $P(w,\overline x), S(w,\overline x)$ and $B(w,\overline x)$ the reproducing kernels of, respectively, $PW(\pi,\mathbf C^{(m)})$, $H^2(S_a,\mathbf C^{(m)})$ and $A^2(S_a,\mathbf C^{(m)}).$

We will show
\begin{align*}
P(w,\overline x)=\frac{1}{(2\pi)^m}\int_{\mathbf R^m} e(w+\overline x,\ulxi)\chi_{B(0;\pi)}(\ulxi)d\ulxi,
\end{align*}
\begin{align*}
S(w,\overline x)=\frac{1}{(2\pi)^m}\int_{\mathbf R^m} e(w+\overline x,\ulxi)e^{-2a|\ulxi|}d\ulxi,
\end{align*}
and
\begin{align*}
B(w,\overline x)=\frac{1}{(2\pi)^m}\int_{\mathbf R^m}e(w+\overline x,\ulxi) \frac{2|\ulxi|}{e^{2a|\ulxi|}-e^{-2a|\ulxi|}} d\ulxi,
\end{align*}
where $e(x,\ulxi)$ is the monogenic exponential function (see \S 2 for details).
In the essence of the Paley-Wiener theorem of $H^2(S_a,\mathbf C^{(m)})$ we give 
\begin{align*}
H^2(S_a,\mathbf C^{(m)})=H^2(\mathbf R^{m+1}_{+,-a},\mathbf C^{(m)})\oplus H^2(\mathbf R^{m+1}_{-,a},\mathbf C^{(m)}),
\end{align*}
where $H^2(\mathbf R^{m+1}_{\pm,\mp a},\mathbf C^{(m)})$ are, respectively, the Hardy spaces on $\mathbf R^{m+1}_{\pm,\mp a}=\{x=x_0+\ulx \in\mathbf R^{{m+1}}: \pm x_0>- a,\ulx\in \mathbf R^m\}.$
Some estimates of $P(w,\overline x),$ $S(w,\overline x)$ and $B(w,\overline x)$ are deduced.

The writing of the paper is organized as follows. In \S 2 notations and terminologies that will be used, as well as an account of the known and relevant results are provided. In \S 3 the spectrum characterizations, representation formulas through the reproducing kernels, are deduced.

 \bigskip
 \section{Preliminaries}
 Denote by $\mathbf R^{(m)}$ ($\mathbf C^{(m)}$) the algebra over the real (complex) number field generated by the basis $\bfe_1,...,\bfe_m$ of $\mathbf R^{m}=\{\ulx=x_1\bfe_1+\cdots+x_m\bfe_m: x_j\in \mathbf R,1\leq j\leq m\},$ where the ${\bfe_j}$'s satisfy the relations $$\bfe_j\bfe_k+ \bfe_k\bfe_j=-2\delta_{jk},\quad j,k=1,...,m,$$ 
 where $\delta_{jk}$ is the Kronecker delta function. We note that $\mathbf R^{(m)}$ ($\mathbf C^{(m)}$) is a particular Clifford algebra with the unit element $\bfe_0 =1.$  

The elements of $\mathbf R^{(m)}$ ($\mathbf C^{(m)}$) are of the form $x=\sum_{T}x_T \bfe_T,$ where $T=\{1\leq j_1<j_2<\cdots<j_l\leq m\}$ runs over all ordered subsets of $\{1,...,m\}$, $x_T\in \mathbf R\  (\mathbf C)$ with $x_{\emptyset}=x_0,$ and $\bfe_T=\bfe_{j_1}\bfe_{j_2}\cdots \bfe_{j_l}$ with the unit element $\bfe_{\emptyset}=\bfe_0=1.$ ${\rm Sc}\{x\}:=x_0$ and ${\rm NSc}\{x\}:=x-\text{Sc }\{x\}$ are respectively called the scalar part and the non-scalar part of $x.$
In this paper, we denote the conjugate of $x\in \mathbf R^{(m)} (\mathbf C^{(m)})$ by $\overline x=\sum_{T}\overline x_T \overline \bfe_T$, where $\overline \bfe_T=\overline \bfe_{j_l}\cdots \overline \bfe_{j_2}\overline \bfe_{j_1}$ with $\overline \bfe_0=\bfe_0$ and $\overline \bfe_j=-\bfe_j$ for $j \neq 0.$ The norm of $x\in \mathbf R^{(m)} (\mathbf C^{(m)})$ is defined as $|x|:=(\text{Sc }\{\overline x x\})^\frac{1}{2}=(\sum_{T}|x_T|^2)^{\frac{1}{2}}.$ $x=x_0+\ulx\in \mathbf R^{m+1}=\{x=x_0+\ulx:x_0\in\mathbf R,\ulx\in \mathbf R^m\}$ is called a para-vector, and the conjugate of a para-vector $x$ is $\overline x=x_0-\ulx.$ If $x$ is a para-vector then $x^{-1}=\frac{\overline x}{|x|^2}.$ For more information about Clifford algebra, we refer to \cite{Brackx-Delanghe-Sommen}.

Let $\Omega$ be an open subset of $\mathbf R^{m+1}.$ A $\mathbf C^{(m)}$-valued function $F$ on $\Omega $ is left-monogenic (resp. right-monogenic) if
\begin{align*}
DF=\sum_{k=0}^m \bfe_k\partial_kF=0\ \left(\text{resp. } FD=\sum_{k=0}^m \partial_kF\bfe_k=0\right),\quad \text{in } \Omega, 
\end{align*}
where $\partial_k=\frac{\partial}{\partial x_k},0\leq k\leq m,$ and $D$ is the Dirac operator. Note that $\overline D(DF)=\Delta F=0$ if $F$ is left-monogenic, which means that each component of a left-monogenic function $F$ is harmonic.
A function that is both left- and right-monogenic is called a monogenic function. Para-vector-valued left-monogenic functions are simultaneously right-monogenic functions, and vice-versa, and thus they are monogenic.

The Fourier transform of a function in $L^1(\mathbf R^m)$ is defined as
$$
\hat f(\ulxi)=\mathcal F(f)(\ulxi)=\int_{\mathbf R^m}e^{-i\langle \ulx,\ulxi\rangle}f(\ulx)d\ulx,
$$
where $\ulxi=\xi_1 \bfe_1+\cdots +\xi_n \bfe_n\in\mathbf R^m,$ and the inverse Fourier transform is formally defined as
$$
 g^\vee(\ulx)=\mathcal F^{-1}(g)(\underline x)=\frac{1}{(2\pi)^m}\int_{\mathbf R^m}e^{i\langle \underline x,\underline \xi \rangle}g(\underline \xi)d\underline \xi.
$$

  The Fourier transformation is linear and thus it, together with some of its properties, can be extended to $\mathbf C^{(m)}$-valued functions. In particular, the Plancherel theorem holds for $\mathbf C^{(m)}$-valued functions: For $\mathbf C^{(m)}$-valued functions $f, g\in L^2(\mathbf R^m,\mathbf C^{(m)})$ there holds
  \begin{align}\label{Plancherel}
  \int_{\mathbf R^m} \overline f(\ulx) g(\ulx)d\ulx = \int_{\mathbf R^m} \overline{\hat{f}}(\ulxi){\hat{g}}(\ulxi)d\ulxi.
  \end{align}

\noindent Define, for $x=x_0+\ulx,$
$$
e(x,\underline \xi)=e^+(x,\ulxi)+e^-(x,\ulxi)
$$
with
$$
e^{\pm}(x,\ulxi)=e^{ i\langle \ulx,\ulxi\rangle}e^{\mp x_0|\ulxi|}\chi_{\pm}(\ulxi),
$$
where $\chi_\pm(\ulxi)=\frac{1}{2}(1\pm i\frac{\ulxi}{|\ulxi|})$ (see e.g. \cite{Li-Mc-Qian}). $\chi_\pm(\ulxi)$ enjoy the projection properties:
\begin{align}\label{important}
\chi_-(\ulxi)\chi_+(\ulxi)=\chi_+(\ulxi)\chi_-(\ulxi)=0,\quad \chi^2_\pm(\ulxi)=\chi_\pm(\ulxi),\quad\chi_+(\ulxi)+\chi_-(\ulxi)=1.
\end{align}

In the following we first state two existing Paley-Wiener theorems in the Clifford algebra setting.
In \cite{Kou-Qian} the following Paley-Wiener theorem is proved.
\begin{prop}[\cite{Kou-Qian}]\label{PW}
Let $f\in L^2(\mathbf R^m,\mathbf C^{(m)}),$ and $R$ a positive number. Then the following two conditions are equivalent:\\
\noindent (i) $f$ may be left-monogenically extended to the whole $\mathbf R^{m+1},$ and there exists a constant $C$ such that $|f(x)|\leq Ce^{R|x|}$ for all $x=x_0+\ulx\in \mathbf R^{m+1}$;\\
\noindent (ii) {\rm supp}$\mathcal F(f)\subset B(0,R)$, where $B(0,R)$ is the ball centered at the origin with radius $R$.
Moreover, if these conditions hold, then
$$
f(x)=\frac{1}{(2\pi)^m}\int_{B(0,R)} e(x,\ulxi)\mathcal F(f)(\ulxi) d\ulxi, \quad x\in \mathbf R^{m+1}.
$$
\end{prop}
By Proposition \ref{PW}, we can define the Paley-Wiener space $PW(\frac{\pi}{h},\mathbf C^{(m)}),h>0,$ as follows. We say $f\in PW(\frac{\pi}{h},\mathbf C^{(m)})$ if $f$ satisfies one of the conditions (i) and (ii) with $R=\frac{\pi}{h}$ in Proposition \ref{PW}. 
 $PW(\frac{\pi}{h},\mathbf C^{(m)})$ is equipped with the inner product
\begin{align*}
\langle f, g\rangle_{PW} = \int_{\mathbf R^m} \overline g(\ulx) f(\ulx) d\ulx, \quad f,g\in PW(\frac{\pi}{h},\mathbf C^{(m)}),
\end{align*} 
and 
\begin{align*}
||f||_{PW}^2 = \text{Sc} (\langle f,f\rangle_{PW}).
\end{align*}
Without loss of generality, we let $h=1.$
Furthermore, the ${\rm sinc}$ function, which is closely related to the reproducing kernel of $PW(\pi,\mathbf C^{(m)}),$ is defined in \cite{Kou-Qian2}, as
\begin{align}\label{sinc}
{\rm sinc}_C(x)=\frac{1}{(2\pi)^m}\int_{\mathbf R^m} e(x,\ulxi)\chi_{[-\pi,\pi]^m}(\ulxi)d\ulxi ,\quad x\in \mathbf R^{m+1},
\end{align}
with the estimation given in the following Lemma.
\begin{lem}[\cite{Kou-Qian}]\label{key-lemma}
There holds
$$
|{\rm sinc}_C(x)|\leq \frac{P(|x_0|)e^{\sqrt{m}\pi |x_0|}}{\prod_{j=1}^m(1+|x_j|)}, \quad x\in \mathbf R^{m+1},
$$
where $P(|x_0|)$ is a polynomial of $|x_0|.$
\end{lem}

The other Paley-Wiener theorem concerns the Hardy space $H^2(\mathbf R^{m+1}_+, \mathbf C^{(m)})$, where 
\begin{align*}
H^2(\mathbf R^{m+1}_+,\mathbf C^{(m)}) = \left \{f \text{ is lef-monogenic in }\mathbf R^{m+1}_+: ||f||_{H^2(\mathbf R^{m+1},\mathbf C^{(m)})} = \sup_{x_0>0}\int_\mathbf{R^m} |f(x_0+\ulx)|^2 d\ulx <\infty \right\}
\end{align*} 
and $\mathbf R^{m+1}_+=\{x=x_0+\ulx\in \mathbf R^{m+1}:x_0>0\},$ which plays a role in our study. The statement is as follows.
\begin{prop}[see e.g. \cite{Gilbert-Murray,Mitrea}]\label{PW-hardy}
 $f\in H^2(\mathbf R^{m+1}_{+}, \mathbf C^{(m)})$ if and only if there exists a measurable function $g$ in $\mathbf R^m$ such that
\begin{align*}
g(\ulxi)\in L^2(\mathbf R^m, \mathbf C^{(m)})
\end{align*}
and
\begin{align*}
f(x) = \frac{1}{(2\pi)^m}\int_{\mathbf R^m} e^{+}(x,\ulxi) g(\ulxi) d\ulxi, \quad x\in \mathbf R^{m+1}_+,
\end{align*}
showing that $g(\ulxi)=\mathcal F(f)(\ulxi).$
\end{prop} 
Furthermore, the $L^p$ version of Proposition \ref{PW-hardy}, $1\leq p\leq \infty$, is stated as follows. Let $\Psi(\mathbf R^m,\mathbf C^{(m)})$ be the Clifford algebra-valued Schwartz space, whose elements are given by
\begin{align*}
\psi(\ulxi)=\sum_{T}\psi_T(\ulxi) \bfe_T,
\end{align*}
where $\psi_T$ are in the Schwartz space $S(\mathbf R^m).$ Denote by $\Psi^\pm(\mathbf R^m,\mathbf C^{(m)})$ the subclasses of $\Psi(\mathbf R^m,\mathbf C^{(m)})$ consisting of the Clifford algebra-valued Schwartz functions of, respectively, the forms
\begin{align*}
\psi(\ulxi)=\psi(\ulxi)\chi_\pm(\ulxi),
\end{align*}
where $\psi(\ulxi)$ takes the zero value in some neighborhood of the origin.
\begin{prop}[see \cite{Dang-Mai-Qian}]
For $f\in H^p(\mathbf R^{m+1}_+,\mathbf C^{(m)})$, $1\leq p\leq \infty,$ there holds
\begin{align*}
(\hat f,\psi) = (f,\hat\psi) =\int_{\mathbf R^m}\hat\psi(\ulx)f(\ulx)d\ulx =0,
\end{align*}
where $\psi\in \Psi^-(\mathbf R^m,\mathbf C^{(m)}).$\\
Conversely, if $f\in L^p(\mathbf R^{m},\mathbf C^{(m)}),1\leq p\leq\infty,$ satisfies $(\hat f,\psi)=0$ for all $\psi\in \Psi(\mathbf R^m,\mathbf C^{(m)})$, then $f(\ulx)$ is the NTBL function of some $f\in H^p(\mathbf R^{m+1}_+,\mathbf C^{(m)}).$
\end{prop}

In this paper we will mainly concern the monogenic Hardy and Bergman spaces on strips.
Denote by $H^p(S_a, \mathbf C^{(m)}),1\leq p<\infty, a>0,$ the monogenic Hardy space on the strip $S_a=\{x\in \mathbf R^{m+1}: |x_0|<a,\ulx\in \mathbf R^m\},$ where
\begin{align*}
H^p(S_a,\mathbf C^{(m)})=\left\{f \text{ is left-monogenic in }S_a: ||f||_{H^p(S_a, \mathbf C^{(m)})}^p=\sup_{|x_0|<a}\int_{\mathbf R^m}|f(x_0+\ulx)|^p d\ulx<\infty \right\}.
\end{align*}
Similarly, we denote by $H^p(\mathbf R^{m+1}_{+,-a},\mathbf C^{(m)})$ the Hardy space consisting of the left-monogenic functions in $\mathbf R^{m+1}_{+,-a}=\{x\in \mathbf R^{m+1}:x_0>-a,\ulx\in \mathbf R^m\}$ satisfying
$$
||f||_{H^p(\mathbf R^{m+1}_{+,-a}, \mathbf C^{(m)})}^p=\sup_{x_0>-a}\int_{\mathbf R^m}|f(x_0+\ulx)|^p d\ulx<\infty,
$$
and by $H^p(\mathbf R^{m+1}_{-,a}, \mathbf C^{(m)})$ the Hardy space consisting of the left-monogenic functions in $\mathbf R^{m+1}_{-,a}=\{x\in \mathbf R^{m+1}:x_0<a,\ulx\in \mathbf R^m\}$ satisfying
$$
||f||_{H^p(\mathbf R^{m+1}_{-,a}, \mathbf C^{(m)})}^p=\sup_{x_0<a}\int_{\mathbf R^m}|f(x_0+\ulx)|^p d\ulx<\infty.
$$

Let $\sigma_m=\frac{\pi^\frac{m+1}{2}}{\Gamma(\frac{m+1}{2})}.$ For $f\in H^p(\mathbf R^{m+1}_{+}, \mathbf C^{(m)}),$ one has the Cauchy integral formula, i.e.,
\begin{prop}[see e.g. \cite{Gilbert-Murray,Mitrea}]\label{Cauchy-formula}
	For $f\in H^p(\mathbf R^{m+1}_+, \mathbf C^{(m)}),1\leq p<\infty,$ we have
	\begin{align*}
	f(x)=\int_{\mathbf R^m}E(x-\uly)f(\uly) d\uly,
	\end{align*}
	where $E(x)=\frac{1}{2\sigma_m}\frac{\overline x}{|x|^{m+1}}$ is the Cauchy kernel, and $f(\uly)$ is the NTBL function of $f$.
\end{prop}

\noindent Denote by $A^p(S_a, \mathbf C^{(m)}),1\leq p<\infty,$ the Bergman spaces on $S_a,$ where
\begin{align}\label{Bergman-norm}
A^p(S_a, \mathbf C^{(m)})=\{f \text{ is left-monogenic in }S_a: ||f||_{A^p(S_a, \mathbf C^{(m)})}^p=\int_{-a}^{a}\int_{\mathbf R^m}|f(x_0+\ulx)|^p d\ulx dx_0<\infty\}.
\end{align}
Similarly, we denote by $A^p(\mathbf R^{m+1}_{+,-a}, \mathbf C^{(m)})$ the Bergman spaces consisting of the left-monogenic functions in $\mathbf R^{m+1}_{+,-a}$ satisfying
$$
||f||_{A^p(\mathbf R^{m+1}_{+,-a}, \mathbf C^{(m)})}^p=\int_{-a}^\infty\int_{\mathbf R^m}|f(x_0+\ulx)|^p d\ulx dx_0<\infty,
$$
and by $A^p(\mathbf R^{m+1}_{-,a}, \mathbf C^{(m)})$ the Bergman spaces consisting of the left-monogenic functions in $\mathbf R^{m+1}_{-,a}$ satisfying
$$
||f||_{A^p(\mathbf R^{m+1}_{-,a}, \mathbf C^{(m)})}^p=\int^{a}_{-\infty}\int_{\mathbf R^m}|f(x_0+\ulx)|^p d\ulx dx_0<\infty.
$$

\section{Monogenic RKHSs and estimations of their reproducing kernels}
\subsection{$PW(\pi, \mathbf C^{(m)})$ as a RKHS}
It is noted that $PW(\frac{\pi}{h}, \mathbf C^{(m)})$ is a RKHS admitting the reproducing kernel given by
$$
P(w,\overline x)=\frac{1}{(2\pi)^m}\int_{\mathbf R^m}e(w+\overline x,\ulxi)\chi_{B(0,\frac{\pi}{h})}(\ulxi)d\ulxi.
$$
In fact, by the Plancherel theorem, Proposition \ref{PW} implies that, for $f\in PW(\pi, \mathbf C^{(m)}),$
\begin{align}\label{verify}
\langle f, P(\cdot,\overline x)\rangle_{PW} = \int_{\mathbf R^m}P(\uly, x) f(\uly)d\uly = \frac{1}{(2\pi)^m}\int_{\mathbf R^m} e(x,\ulxi)\chi_{B(0,\pi)}\mathcal F(f)(\ulxi) d\ulxi =f(x),
\end{align}
which shows that $P(w,\overline x)$ is the reproducing kernel for $PW(\pi,\mathbf C^{(m)}).$

Next we induce another ${\rm sinc}$ function in $PW(\pi,\mathbf C^{(m)})$ by
\begin{align}\label{sinc-B}
{\rm sinc}_{B}(x)=P(x,0)=\frac{1}{(2\pi)^m}\int_{\mathbf R^m}e(x,\ulxi)\chi_{B(0,\pi)}(\ulxi)d\ulxi,\quad x\in \mathbf R^{m+1}.
\end{align}
 The following estimation of ${\rm sinc}_B$ is analogous to that of ${\rm sinc}_C$ given in Lemma \ref{key-lemma}. Moreover, the sinc function ${\rm sinc}_B$ has more significance due to its relation with the reproducing kernel of $PW(\pi,\mathbf C^{(m)})$ through (\ref{verify}) and (\ref{sinc-B}).
\begin{lem}\label{PW-infty}
$$
|{\rm sinc}_B(x)|\leq M\frac{(1+|x_0|)e^{|x_0|\pi}}{|\ulx|^\frac{m+1}{2}}, \quad x\in \mathbf R^{m+1}, \ |\ulx|\geq 1,
$$
where $M$ is a constant.
\end{lem}

\begin{proof}
Observing that
\begin{align*}
{\rm sinc}_{B}(x)&=\frac{1}{(2\pi)^m}\int_{\mathbf R^m}e(x,\ulxi)\chi_{B(0,\pi)}(\ulxi)d\ulxi\\
&=\frac{1}{(2\pi)^m}\int_{\mathbf R^m}e^+(x,\ulxi)\chi_{B(0,\pi)}(\ulxi)d\ulxi+\frac{1}{(2\pi)^m}\int_{\mathbf R^m}e^-(x,\ulxi)\chi_{B(0,\pi)}(\ulxi)d\ulxi\\
&={\rm sinc}_B^+(x)+{\rm sinc}_B^-(x).
\end{align*}
We are thus reduced to estimate ${\rm sinc}_B^+(x)$ and ${\rm sinc}_B^-(x)$ separately.

Let $d\sigma(\ulxi^\prime)$ be the area element of the $(m-1)$-sphere $S^{m-1}$.
For $x=x_0+\ulx\in \mathbf R^{m+1},$ we have
\begin{align}\label{tem_estimate}
\begin{split}
& {\rm sinc}^{\pm}_B(x)\\
&=\frac{1}{(2\pi)^m}\int_{B(0,\pi)}e^{i\langle \ulx,\ulxi\rangle}e^{\mp x_0|\ulxi|}\chi_\pm (\ulxi)d\ulxi\\
&=\frac{1}{2(2\pi)^m}\int_0^\pi\int_{S^{m-1}} e^{ir\langle \ulx,\ulxi^\prime\rangle}e^{\mp x_0r}(1\pm i\ulxi^\prime)r^{m-1}d\sigma(\ulxi^\prime)dr\\
&=\frac{1}{2(2\pi)^m}\left(\int_0^\pi\int_{S^{m-1}} e^{ir\langle \ulx,\ulxi^\prime\rangle}e^{\mp x_0r}r^{m-1}d\sigma(\ulxi^\prime)dr \pm \int_0^\pi\int_{S^{m-1}} e^{ir\langle \ulx,\ulxi^\prime\rangle}e^{\mp x_0r}(i\ulxi^\prime)r^{m-1}d\sigma(\ulxi^\prime)dr\right)\\
&=\frac{1}{2(2\pi)^m}(I_1 \pm I_2).
\end{split}
\end{align}
First consider $I_1.$ We have
\begin{align}\label{I1}
\begin{split}
I_1 &=\int_{0}^\pi e^{\mp x_0r}r^{m-1}\int_{S^{m-1}} e^{ir\langle \ulx,\ulxi^\prime\rangle}d\sigma(\ulxi^\prime)dr\\
&=\int_{0}^\pi e^{\mp x_0r}r^{m-1}\int_{S^{m-1}} e^{ir\langle U\ulx,U\ulxi^\prime\rangle}d\sigma(\ulxi^\prime)dr\\
&=\int_{0}^\pi e^{\mp x_0r}r^{m-1}\int_{0}^\pi e^{ir|\ulx|\cos\theta}(\sin\theta)^{m-2} d\theta dr\\
&=\int_{0}^\pi e^{\mp x_0r}r^{m-1}\int_{-1}^1 e^{ir|\ulx|\eta}(1-\eta^2)^\frac{m-3}{2} d\eta dr,
\end{split}
\end{align}
where $U\in \mathbf O(m)=\{A\in \mathbf{GL}(m); \langle A\ulx,A\ulxi\rangle=\langle \ulx,\ulxi \rangle,\ulx,\ulxi\in \mathbf R^m\}$ is a rotation fixing the origin and making $U\ulx=|\ulx|e_1.$  In the change of variable we used $d\sigma(\ulxi^\prime)=d\sigma(U\ulxi^\prime).$
We recall that
\begin{align*}
\int_{-1}^1 e^{ir|\ulx|\eta}(1-\eta^2)^\frac{m-3}{2}d\eta &=\omega_\frac{m-2}{2} (r|\ulx|)^{-\frac{m-2}{2}}J_{\frac{m-2}{2}}(r|\ulx|),
\end{align*}
where $\omega_\frac{m-2}{2}=\Gamma(\frac{m-1}{2})\Gamma(\frac{1}{2}),$ and $J_k(t)$ is the Bessel function given by $$J_k(t)=\frac{(\frac{t}{2})^k}{\omega_k}\int_{-1}^1 e^{its}(1-s^2)^{\frac{2k-1}{2}}ds,\quad k>-\frac{1}{2}.$$
We also need the following properties of $J_k(t)$ (see \cite{SW}):
\begin{align}\label{property-1}
\frac{d}{dt}(t^k J_k(\alpha t))=\alpha t^k J_{k-1}(\alpha t)
\end{align}
and
\begin{align}\label{property-2}
J_k(t)=O(t^{-\frac{1}{2}}) \quad {\text{as } } \quad t\to\infty.
\end{align}
Since
\begin{align*}
I_1 &=\omega_\frac{m-2}{2}\int_0^\pi e^{\mp x_0r}r^{m-1}(r|\ulx|)^{-\frac{m-2}{2}} J_{\frac{m-2}{2}}(r|\ulx|) dr\\
&=\frac{\omega_\frac{m-2}{2}}{|\ulx|^{\frac{m}{2}}}\int_0^\pi e^{\mp x_0r}r^{\frac{m}{2}}|\ulx|J_{\frac{m}{2}-1}(r|\ulx|)dr\\
&= \frac{\omega_\frac{m-2}{2}}{|\ulx|^{\frac{m}{2}}}\left(r^{\frac{m}{2}}J_{\frac{m}{2}}(r|\ulx|)e^{\mp x_0 r}\big|_{0}^\pi-{\mp x_0}\int_0^\pi e^{\mp x_0r}r^{\frac{m}{2}}J_{\frac{m}{2}}(r|\ulx|)dr\right),
\end{align*}
by (\ref{property-2}), there exists a constant $C_1>0$ such that
\begin{align*}
|I_1|\leq C_1\frac{(1+|x_0|)e^{|x_0|\pi}}{|\ulx|^\frac{m+1}{2}}.
\end{align*}

Next we consider $I_2.$ As in (\ref{I1}), we have
\begin{align}\label{I2}
\begin{split}
I_2 &=i\int_0^\pi e^{\mp x_0 r}r^{m-1}\int_{S^{m-1}}e^{ir\langle U\ulx,U\ulxi^\prime\rangle} \ulxi^\prime d\sigma(\ulxi^\prime) dr\\
&=i\int_0^\pi e^{\mp x_0 r}r^{m-1}\int_{S^{m-1}}e^{ir\langle U\ulx,\ulzeta^\prime\rangle} U\ulzeta^\prime d\sigma(\ulzeta^\prime) dr\\
&=i\int_0^\pi e^{\mp x_0 r}r^{m-1}\int_{0}^{2\pi}\int_0^\pi\cdots\int_{0}^\pi e^{ir|\ulx|\cos\theta_1} U\ulzeta^\prime (\sin\theta_1)^{m-2}\cdots \sin\theta_{m-2}d\theta_1\cdots d\theta_{m-2}d\theta_{m-1} dr,
\end{split}
\end{align}
where we write $d\sigma(\zeta^\prime)$ in spherical coordinates, and $U\in \mathbf O(m)$ fixes the origin such that $U\ulx=|\ulx|e_1.$ For $\ulzeta^\prime\in S^{m-1}$ and $U=(u_{jk})_{m\times m}\in \mathbf O(m)$, we have
$$
U\ulzeta^\prime=(u_{11}\cos\theta_1+V_1(\theta_2,...,\theta_{m-1})\sin\theta_1,\dots,u_{m1}\cos\theta_1+V_m(\theta_2,...,\theta_{m-1})\sin\theta_1)^T,
$$
where $V_j(\theta_2,...,\theta_{m-1})$ depends on $(\theta_2,...,\theta_m),$ and $|V_j(\theta_2,...,\theta_{m-1})|\leq \sum_{k=2}^m|u_{jk}|\leq m.$ Therefore, to estimate $I_2,$ it suffices to estimate $I_{21}$ and $I_{22}$, where
$$
I_{21}=\int_0^\pi e^{ir|\ulx|\cos\theta_1} \cos\theta_1 (\sin \theta_1)^{m-2}d\theta_1
$$
and
$$
I_{22}=\int_0^\pi e^{ir|\ulx|\cos\theta_1} \sin\theta_1 (\sin \theta_1)^{m-2}d\theta_1.
$$
Similarly, we have that
\begin{align*}
I_{21}=\int_{-1}^1 e^{ir|\ulx|\eta}\eta (1-\eta^2)^\frac{m-3}{2} d\eta &=-\frac{1}{m-1}(1-\eta^2)^{\frac{m-1}{2}}e^{ir|\ulx|\eta}\big|_{-1}^1+\frac{ir|\ulx|}{m-1}\int_{-1}^1e^{ir|\ulx|\eta}(1-\eta^2)^\frac{m-1}{2}d\eta\\
&=\frac{i\omega_\frac{m}{2}}{m-1} (r|\ulx|)^{-\frac{m}{2}+1}J_{\frac{m}{2}}(r|\ulx|),
\end{align*}
and
\begin{align*}
I_{22}=\int_{-1}^1 e^{ir|\ulx|\eta}(1-\eta^2)^\frac{m-2}{2}d\eta=\omega_\frac{m-1}{2}(r|\ulx|)^{-\frac{m-1}{2}}J_{\frac{m-1}{2}}(r|\ulx|).
\end{align*}
For $I_{21},$ we have that
\begin{align*}
&\int_{0}^\pi e^{\mp x_0 r}r^{m-1} (r|\ulx|)^{-\frac{m}{2}+1}J_{\frac{m}{2}}(r|\ulx|)dr \\
&=\frac{1}{|\ulx|^{\frac{m}{2}}}\int_0^\pi e^{\mp x_0r} r^{\frac{m}{2}}|\ulx| J_{\frac{m}{2}}(r|\ulx|)dr\\
&=\frac{1}{|\ulx|^\frac{m}{2}}\left(r^{\frac{m}{2}+1}J_{\frac{m}{2}+1}(r|\ulx|)\frac{e^{\mp x_0r}}{r}\big|_0^\pi - \int_0^\pi \frac{d(\frac{e^{\mp x_0r}}{r})}{dr}r^{\frac{m}{2}+1}J_{\frac{m}{2}+1}(r|\ulx|)dr\right)\\
&=\frac{1}{|\ulx|^\frac{m}{2}}\left(\pi^{\frac{m}{2}}J_{\frac{m}{2}+1}(\pi|\ulx|){e^{\mp x_0\pi}} - \int_0^\pi (\mp x_0e^{\mp x_0r}r-e^{\mp x_0r})r^{\frac{m}{2}-1}J_{\frac{m}{2}+1}(r|\ulx|)dr\right),
\end{align*}
and then have
\begin{align*}
&\left|\int_{0}^\pi e^{\mp x_0 r}r^{m-1} (r|\ulx|)^{-\frac{m}{2}+1}J_{\frac{m}{2}}(r|\ulx|)dr \right|\\
&\leq \frac{C_2^\prime}{|\ulx|^{\frac{m}{2}+1}} +\frac{C_2^{\prime\prime}}{|\ulx|^{\frac{m}{2}+1}}\int_0^\pi (|x_0|e^{\mp x_0r}r^{\frac{m-1}{2}}+e^{\mp x_0r}r^{\frac{m-3}{2}})dr.
\end{align*}
Note that
\begin{align*}
\int_0^\pi e^{\mp x_0r}r^\frac{m-3}{2}dr<\infty, \quad m\geq 3,
\end{align*}
and for $m=2,$ the same conclusion can be given by integration by parts, i.e.,
\begin{align*}
\int_0^\pi e^{\mp x_0r}r^{-\frac{1}{2}}dr=2{r^\frac{1}{2}}e^{\mp x_0r}\big|_0^\pi-(\mp 2x_0)\int_0^\pi e^{\mp x_0 r}r^\frac{1}{2}dr <\infty.
\end{align*}
Thus
\begin{align*}
&\left|\int_{0}^\pi e^{\mp x_0 r}r^{m-1} (r|\ulx|)^{-\frac{m}{2}+1}J_{\frac{m}{2}}(r|\ulx|)dr \right|\leq \frac{C_2(1+|x_0|)e^{|x_0|\pi}}{|\ulx|^{\frac{m}{2}+1}},
\end{align*}
where $C_2$ is a constant.
For $I_{22}$, we first consider the case $m=2,$ and have
\begin{align*}
\left|\int_{0}^\pi e^{\mp x_0 r}r I_{22}dr \right|=\left|\frac{\pi}{2}\int_0^\pi e^{\mp x_0r}r \frac{e^{ir|\ulx|}-e^{-ir|\ulx|}}{ir|\ulx|} dr \right|
&\leq C_2\frac{e^{|x_0|\pi}}{|\ulx|(x_0^2+|\ulx|^2)^\frac{1}{2}}\leq C_3\frac{e^{|x_0|\pi}}{|\ulx|^2}.
\end{align*}
For $m\geq 3,$ using integration by parts, we have
\begin{align*}
&\left|\int_0^\pi e^{\mp x_0 r}r^{m-1} (r|\ulx|)^{-\frac{m-1}{2}} J_{\frac{m-1}{2}}(r\ulx)dr \right| \\
&=\frac{1}{|\ulx|^\frac{m+1}{2}}\left|\int_0^\pi \frac{e^{\mp x_0 r}}{r}|\ulx|r^\frac{m+1}{2} J_{\frac{m-1}{2}}(r\ulx)dr\right|\\
&\leq \frac{1}{|\ulx|^\frac{m+1}{2}}\left|r^\frac{m+1}{2} J_{\frac{m+1}{2}}(r|\ulx|)\frac{e^{\mp x_0r}}{r}\big|_{0}^\pi \right| +\frac{1}{|\ulx|^\frac{m+1}{2}} \left|\int_{0}^\pi r^\frac{m+1}{2} J_{\frac{m+1}{2}}(r|\ulx|) \frac{\mp x_0e^{\mp x_0r}r-e^{\mp x_0r}}{r^2} dr\right|\\
&\leq C_4\frac{(1+|x_0|)e^{|x_0|\pi}}{|\ulx|^\frac{m+2}{2}}.
\end{align*}
Therefore, when $|\ulx|\geq 1,$
$$
|I_{2}|\leq C_5 \frac{(1+|x_0|)e^{|x_0|\pi}}{|\ulx|^\frac{m+1}{2}}.
$$
We thus obtain the desired result.
\end{proof}

\subsection{$H^2(S_a,\mathbf C^{(m)})$ as a RKHS}
The space $H^2(S_a, \mathbf C^{(m)})$ is also a Paley-Wiener type RKHS. In this part we will first prove the Paley-Wiener theorem for $H^2(S_a,\mathbf C^{(m)}),$ and then construct its Szeg\"o kernel.

It is well-known that (see e.g. \cite{Gilbert-Murray,Dang-Mai-Qian})
$$
L^2(\mathbf R^m, \mathbf C^{(m)})=H^2(\mathbf R^{m+1}_+, \mathbf C^{(m)})\oplus H^2(\mathbf R^{m+1}_-, \mathbf C^{(m)}).
$$
 There exists a similar decomposition for $H^2(S_a, \mathbf C^{(m)}).$ In the following we give
\begin{thm}\label{PW-strips}
Let $f\in L^2(\mathbf R^m, \mathbf C^{(m)}).$ Then $f$ is the restriction to
$\mathbf R^m$ of a function in
$H^2(S_a, \mathbf C^{(m)})$ if and only if there exists a measurable function
 $g$ in $\mathbf R^m$ such that
\begin{align}\label{cond1}
e^{a|\ulxi|}g(\ulxi)\in L^2(\mathbf R^m, \mathbf C^{(m)})
\end{align}
and
\begin{align}\label{cond2}
f(x)=\frac{1}{(2\pi)^m}\int_{\mathbf R^m}e(x,\ulxi)g(\ulxi)d\ulxi,\quad x\in S_a,
\end{align}
showing that $g(\ulxi)=\mathcal F f(\ulxi).$
Moreover, there exist $f_+\in H^2(\mathbf R^{m+1}_{+,-a}, \mathbf C^{(m)})$ and $f_-\in H^2(\mathbf R^{m+1}_{-,a}, \mathbf C^{(m)})$ such that
\begin{align}\label{Decomposition-2}
f(x)=f_+(x)+f_-(x),\quad x\in S_a,
\end{align}
where the above decomposition is unique, and implies
$$
H^2(S_a,\mathbf C^{(m)})=H^2(\mathbf R^{m+1}_{+,-a}, \mathbf C^{(m)})\oplus H^2(\mathbf R^{m+1}_{-,a}, \mathbf C^{(m)}).
$$
\end{thm}

To prove Theorem \ref{PW-strips}, we first recall the Paley-Wiener theorem
for $H^2(\mathbf R^{m+1}_{+,-a}, \mathbf C^{(m)})$ and $H^2(\mathbf R^{m+1}_{-,a}, \mathbf C^{(m)}).$
We can only consider the case for $H^2(\mathbf R^{m+1}_{+,-a}, \mathbf C^{(m)}),
$ as the case for $H^2(\mathbf R^{m+1}_{-,a}, \mathbf C^{(m)})$ is similar. In fact, as a consequence of the Paley-Wiener theorem for $H^2(\mathbf R^{m+1}_+, \mathbf C^{(m)})$ (see e.g. \cite{Gilbert-Murray,Mitrea}, and see also Proposition \ref{PW-hardy}), we have
\begin{lem}\label{PW-Upper-Hardy}
 $f\in H^2(\mathbf R^{m+1}_{+,-a},\mathbf C^{(m)})$ if and only if there exists a measurable function $g$ in $\mathbf R^m$ such that
 \begin{align*}
 e^{a|\ulxi|}g(\ulxi)\in L^2(\mathbf R^m, \mathbf C^{(m)})
 \end{align*}
 and
 \begin{align}\label{PW-a}
 f(x) = \frac{1}{(2\pi)^m}\int_{\mathbf R^m} e^{+}(x,\ulxi) g(\ulxi) d\ulxi, \quad x\in \mathbf R^{m+1}_{+,-a},
 \end{align}
 showing that $g(\ulxi)=\mathcal F(f)(\ulxi).$
\end{lem}
\begin{proof}
For $f\in H^2(\mathbf R^{m+1}_{+,-a}, \mathbf C^{(m)}),$ we set $F(x)=f(-a+x_0+\ulx)$ with $x_0>0,$ and then $F(x)\in H^2(\mathbf R^{m+1}_+, \mathbf C^{(m)})$. Applying Proposition \ref{PW-hardy} to $F(x),$ we can get the desired relation.
\end{proof}

In the following we prove Theorem \ref{PW-strips}.

\begin{proof}
We first assume that (\ref{cond1}) holds. The monogenicity of $f(x)$ defined through (\ref{cond2}) follows from (\ref{PW-a}). Then, by Plancherel's theorem we have, for $x\in S_a,$
\begin{align}\label{Plancherel}
\begin{split}
\int_{\mathbf R^m}|f(x_0+\ulx)|^2d\ulx &=\frac{1}{(2\pi)^m} \int_{\mathbf R^m} \left|(e^{-x_0|\ulxi|}\chi_+(\ulxi)+e^{x_0|\ulxi|}\chi_-(\ulxi))g(\ulxi)\right|^2 d\ulxi\\
&= \frac{1}{(2\pi)^m}\int_{\mathbf R^m} \left|(e^{-x_0|\ulxi|}\chi_+(\ulxi))g(\ulxi)\right|^2 d\ulxi+  \int_{\mathbf R^m} \left|(e^{x_0|\ulxi|}\chi_-(\ulxi))g(\ulxi)\right|^2 d\ulxi\\
&<\infty,
\end{split}
\end{align}
where the second equality is a consequence of the orthogonality
(\ref{important}), and the last inequality follows from the assumption (\ref{cond1}).

\noindent Next we assume that $f\in H^2(S_a, \mathbf C^{(m)}).$
 Let $\phi\in S(\mathbf R^m)$ be a scalar-valued
 Schwarz function with $\int_{\mathbf R^m}\phi(\ulxi) d\ulxi=1,$
 where $\mathcal F(\phi)$ has compact support and is
 equal to $1$ in the unit ball $B(0,1)$.
 Set $\phi_\epsilon(\ulx)=\frac{1}{\epsilon^m}
 \phi(\frac{\ulx}{\epsilon}),\epsilon>0,$ and then
 $\mathcal F(\phi_\epsilon)(\ulxi)=\mathcal F(\phi)(\epsilon \ulxi).$
 Since $\phi\in S(\mathbf R^m),$ we have
 $\psi(\ulx)=\esssup_{|\ulxi|\geq |\ulx|} |\phi(\ulxi)|\in L^1(\mathbf R^m).$
 Thus $\phi_\epsilon$ is an approximation to the identity \cite[page 13]{SW}.
 We define
$$
g_\epsilon(x_0+\ulx)=\left(f(x_0+\cdot)*\phi_\epsilon\right)(\ulx).
$$
Taking Fourier transform to the both sides, we have
$
\mathcal F(g_\epsilon(x_0+\cdot))=
\mathcal F(f(x_0+\cdot))\mathcal F(\phi_\epsilon),
$
showing that for each fixed $x_0$ the set
${\rm supp} \mathcal F(g_\epsilon)$ in $\mathbf R^m$ is compact.
The function $g_\epsilon(x)$ is left-monogenic and satisfies
$$
\sup_{|x_0|<a}||g_\epsilon(x_0+\cdot)||_{L^2(\mathbf R^m, \mathbf C^{(m)})}\leq C \sup_{|x_0|<a}||f(x_0+\cdot)||_{L^2(\mathbf R^m, \mathbf C^{(m)})} ||\phi_\epsilon||_{L^1(\mathbf R^m, \mathbf C^{(m)})}<\infty.
$$
Hence
\begin{align}\label{tem1}
G_\epsilon(x)=\frac{1}{(2\pi)^m}\int_{\mathbf R^m}e(x,\ulxi)\mathcal F(g_\epsilon)(\ulxi)d\ulxi, \quad x\in S_a,
\end{align}
is well-defined due to compactness of ${\rm supp}\mathcal F(g_\epsilon).$
 In particular, $\mathcal F(G_\epsilon)=\mathcal F(g_\epsilon),$
 and $G_\epsilon$ and $g_\epsilon$ are both left-monogenic in $S_a.$ Note that the two left-monogenic functions, $G_\epsilon$ and $g_\epsilon,$ defined in $S_a,$ have common values
 on $\mathbf R^m,$ and thus have to be identical (see e.g. \cite{Pena,Brackx-Delanghe-Sommen}).
   Therefore,
\begin{align*}
\mathcal F(f(x_0+\cdot))(\ulxi)\mathcal F(\phi_\epsilon)(\ulxi)&=\mathcal F(g_\epsilon(x_0+\cdot))(\ulxi)=(e^{-x_0|\ulxi|}\chi_+(\ulxi)+e^{x_0|\ulxi|}\chi_-(\ulxi))\mathcal F(g_\epsilon)(\ulxi)\\
&=(e^{-x_0|\ulxi|}\chi_+(\ulxi)+e^{x_0|\ulxi|}\chi_-(\ulxi))\mathcal F(f)(\ulxi)\mathcal F(\phi_\epsilon)(\ulxi).
\end{align*}
Thus
\begin{align}\label{tem2}
\mathcal F(f(x_0+\cdot))(\ulxi)=(e^{-x_0|\ulxi|}\chi_+(\ulxi)+e^{x_0|\ulxi|}\chi_-(\ulxi))\mathcal F(f)(\ulxi)
\end{align}
for $\ulxi\in B(0,\frac{1}{\epsilon}).$ Since $\epsilon>0$ is arbitrary,
we see that (\ref{tem2}) holds for all $\ulxi\in\mathbf R^m.$
Replacing $x_0$ by $-x_0$ in (\ref{tem2}), we have
\begin{align}\label{tem3}
\mathcal F(f(-x_0+\cdot))(\ulxi)=(e^{x_0|\ulxi|}\chi_+(\ulxi)+e^{-x_0|\ulxi|}\chi_-(\ulxi))\mathcal F(f)(\ulxi).
\end{align}
Consequently, we have
\begin{align}\label{tem5}
\mathcal F(f(x_0+\cdot))(\ulxi)+\mathcal F(f(-x_0+\cdot))(\ulxi)=(e^{-x_0|\ulxi|}+e^{x_0|\ulxi|})\mathcal F(f)(\ulxi).
\end{align}
By Plancherel's theorem and $f\in H^2(S_a,\mathbf C^{(m)}),$ we have
\begin{align*}
\frac{1}{(2\pi)^m}\int_{\mathbf R^m} |(e^{-x_0|\ulxi|}+e^{x_0|\ulxi|})\mathcal F(f)(\ulxi)|^2 d\ulxi&= \frac{1}{(2\pi)^m}\int_{\mathbf R^m} |\mathcal F(f(x_0+\cdot))(\ulxi)+\mathcal F(f(-x_0+\cdot))(\ulxi)|^2d\ulxi\\
&= \int_{\mathbf R^m} |f(x_0+\ulx)+f(-x_0+\ulx)|^2 d\ulx\\
&\leq C ||f||_{H^2(S_a,\mathbf C^{(m)})}^2\\
&<\infty,
\end{align*}
which gives $e^{a|\ulxi|}\mathcal F(f)(\ulxi)\in L^2(\mathbf R^m, \mathbf C^{(m)}).$
  By applying the Lebesgue dominated convergence theorem to (\ref{tem1}), we can obtain
\begin{align*}
f(x) = \frac{1}{(2\pi)^m}\int_{\mathbf R^m} e(x,\ulxi)\mathcal F(f)(\ulxi)d\ulxi.
\end{align*}
The conditions for using the Lebesgue dominated convergence theorem are verified as follows. By the definition of $g_\epsilon,$ we have
\begin{align*}
\lim_{\epsilon\to 0} e(x,\ulxi)\mathcal F(g_\epsilon)(\ulxi) &=\lim_{\epsilon\to 0} e(x,\ulxi)\mathcal F(f)(\ulxi)\mathcal F(\phi_{\epsilon})(\ulxi)\\
 &=\lim_{\epsilon\to 0} e(x,\ulxi)\mathcal F(f)(\ulxi)\mathcal F(\phi)(\epsilon\ulxi)\\
 & = e(x,\ulxi)\mathcal F(f)(\ulxi), \quad \text{ a.e. } \ulxi\in \mathbf R^m,
\end{align*}
\begin{align*}
|e(x,\ulxi)\mathcal F(g_\epsilon)(\ulxi)|\leq |e(x,\ulxi)\mathcal F(f)(\ulxi)\mathcal F(\phi)(\epsilon \ulxi)|\leq |e(x,\ulxi)\mathcal F(f)(\ulxi)| ||\phi||_{L^1(\mathbf R^m)}
\end{align*}
and
\begin{align*}
\int_{\mathbf R^m}|e(x,\ulxi)\mathcal F(f)(\ulxi)| d\ulxi \leq \left(\int_{\mathbf R^m} e^{-2a|\ulxi|}|e(x,\ulxi)|^2d\ulxi\right)^\frac{1}{2}\left(\int_{\mathbf R^m}e^{2a|\ulxi|}|\mathcal F(f)(\ulxi)|^2d\ulxi\right)^\frac{1}{2}<\infty.
\end{align*}

To complete the proof we need to show uniqueness of
the decomposition (\ref{Decomposition-2}).
 In fact, if there exist $h_+\in H^2(\mathbf R^{m+1}_{+,-a},\mathbf C^{(m)})$
 and $h_-\in H^2(\mathbf R^{m+1}_{-,a},\mathbf C^{(m)})$ such that $f=h_++h_-,$ then
we have $h=f_+-h_+=h_--f_-\in
H^2(\mathbf R^{m+1}_{{+,-a}},\mathbf C^{(m)})\cap H^2(\mathbf R^{m+1}_{-,a},\mathbf C^{(m)}).$ This
indeed implies $h=0$ since $H^2(\mathbf R^{m+1}_{{+,-a}},\mathbf C^{(m)})
\cap H^2(\mathbf R^{m+1}_{-,a},\mathbf C^{(m)})\subset H^2(\mathbf R^{m+1}_{{+}},\mathbf C^{(m)})
\cap H^2(\mathbf R^{m+1}_{-},\mathbf C^{(m)})=\{0\}.$
\end{proof}

\begin{remark} {\bf 1} \label{1} We can identify $H^2(S_a, \mathbf C^{(m)})$ with the closed subspace of
$L^2(\mathbf R^m, \mathbf C^{(m)}):$
$$H^2_a(\mathbf R^m, \mathbf C^{(m)})=\{g\in L^2(\mathbf R^m, \mathbf C^{(m)})\ :
\ e^{a|\ulxi|}g(\ulxi)\in L^2(\mathbf R^m, \mathbf C^{(m)})\}.$$
Let $s_x(\ulxi)=e^{-a|\ulxi|}e(\overline x,\ulxi).$ It is obvious that $s_x\in H^2_a(\mathbf R^m, \mathbf C^{(m)}).$
By Theorem \ref{PW-strips}, we have
\begin{align*}
f(x)=\langle f_a,s_x\rangle_{L^2(\mathbf R^m,\mathbf C^{(m)})},\quad \text{for }f\in H^2(S_a, \mathbf C^{(m)}),
\end{align*}
where $f_a$ is one associated with $f$ in $H^2_a(\mathbf R^m, \mathbf C^{(m)}).$
Then we have an induced inner product on $H^2(S_a, \mathbf C^{(m)})$ defined by
\begin{align*}
\langle f,h \rangle_{H^2(S_a, \mathbf C^{(m)})} =
\langle f_a,h_a\rangle_{L^2(\mathbf R^m, \mathbf C^{(m)})},\quad \text{for }f,g\in H^2(S_a, \mathbf C^{(m)}),
\end{align*}
where $f_a$ and $h_a,$ respectively, correspond to $f$ and $h$ in $H^2_a(\mathbf R^m, \mathbf C^{(m)}).$
Accordingly, the reproducing kernel $S(w,\overline x)$ for $H^2(S_a, \mathbf C^{(m)})$ in the above induced norm is given by
\begin{align*}
S(w,\overline x)&=\langle s_x, s_w\rangle_{L^2(\mathbf R^m,\mathbf C^{(m)})}\\
&=\frac{1}{(2\pi)^m}\int_{\mathbf R^m}e^{-2a|\ulxi|}e(w+\overline x,\ulxi)d\ulxi\\
&=\frac{1}{(2\pi)^m}\int_{\mathbf R^m}e^{-2a|\ulxi|}e^+(w+\overline x,\ulxi)d\ulxi + \frac{1}{(2\pi)^m}\int_{\mathbf R^m}e^{-2a|\ulxi|}e^-(w+\overline x,\ulxi)d\ulxi\\
&=\frac{1}{2\sigma_m}\frac{w+\overline x+2a}{|w+\overline x+2a|^{m+1}}-\frac{1}{2\sigma_m}\frac{w+\overline x-2a}{|w+\overline x-2a|^{m+1}}\\
&=S_{+,-a}(w,\overline x)+S_{-,a}(w,\overline x),
\end{align*}
where $S_{+,-a}(w,\overline x)$ and $S_{-,a}(w,\overline x)$
are the Szeg\"o kernels for $H^2(\mathbf R^{m+1}_{+,-a}, \mathbf C^{(m)})$
and $H^2(\mathbf R^{m+1}_{-,a}, \mathbf C^{(m)})$, respectively.\end{remark}

\begin{remark} {\bf 2}\label{2}
We note that Theorem \ref{PW-strips} can be generalized to $H^p(S_a, \mathbf C^{(m)}),1\leq p\leq 2,$  stated as
\begin{thm}\label{PW-strips-g}
	Suppose that $1\leq p\leq 2.$ If $g(\ulxi)\in L^q(\mathbf R^m, \mathbf C^{(m)}),q=\frac{p}{p-1},$ and $e^{a|\ulxi|}\chi_+(\ulxi)g(\ulxi)$ and $e^{a|\ulxi|}\chi_-(\ulxi)g(\ulxi)$ are the Fourier transforms of some functions in $L^p(\mathbf R^m)$, then
	\begin{align}\label{cond2-g}
	f(x)=\frac{1}{(2\pi)^m}\int_{\mathbf R^m}e(x,\ulxi)g(\ulxi)d\ulxi,\quad x\in S_a,
	\end{align}
	is in $H^p(S_a, \mathbf C^{(m)}).$ Moreover, there exist $f_+\in H^p(\mathbf R^{m+1}_{+,-a}, \mathbf C^{(m)})$ and $f_-\in H^p(\mathbf R^{m+1}_{-,a}, \mathbf C^{(m)})$ such that
	\begin{align}\label{Decomposition-g}
	f(x)=f_+(x)+f_-(x),\quad x\in S_a,
	\end{align}
	where the above decomposition is unique.
	
	Conversely, if $f\in H^p(S_a, \mathbf C^{(m)}),1\leq p\leq 2,$ then $g(\ulxi)=\mathcal F(f)(\ulxi)$ such that
	$$
	e^{a|\ulxi|}g(\ulxi)\in L^q(\mathbf R^m, \mathbf C^{(m)})
	$$
	and (\ref{cond2-g}) holds.

\end{thm}
\begin{proof}
	For $p=2$ the result follows from Theorem \ref{PW-strips}. In the following we only need to consider $1\leq p<2.$ \\
	We first assume that $g(\ulxi)\in L^q(\mathbf R^m, \mathbf C^{(m)})$ such that there exist $g_+(\ulx),g_-(\ulx)\in L^p(\mathbf R^m, \mathbf C^{(m)})$ satisfying
	$
	e^{a|\ulxi|}\chi_+(\ulxi) g(\ulxi) = \mathcal F(g_+)(\ulxi)
	$
	and
	$
	e^{a|\ulxi|}\chi_-(\ulxi) g(\ulxi) = \mathcal F(g_-)(\ulxi).
	$
We define
	\begin{align*}
	f(x) &= \frac{1}{(2\pi)^m}\int_{\mathbf R^m}e(x,\ulxi)g(\ulxi)d\ulxi \\
	&= \frac{1}{(2\pi)^m}\int_{\mathbf R^m} e^+(x,\ulxi)g(\ulxi)d\ulxi + \frac{1}{(2\pi)^m}\int_{\mathbf R^m} e^-(x,\ulxi)g(\ulxi)d\ulxi\\
	&=f_+(x)+f_-(x).
	\end{align*}

	For $1<p<2,$ we have
	\begin{align*}
	f_+(x) & = \frac{1}{(2\pi)^m}\int_{\mathbf R^m} e^{i\langle \ulx,\ulxi\rangle}e^{-x_0|\ulxi|}\chi_+(\ulxi)g(\ulxi)d\ulxi \\
	& = \frac{1}{(2\pi)^m}\int_{\mathbf R^m} e^{i\langle \ulx,\ulxi\rangle}e^{-(a+x_0)|\ulxi|}\chi_+(\ulxi)\mathcal F(g_+)(\ulxi)d\ulxi \\
	& = \int_{\mathbf R^m} {S_{+,-a}(-\ulw, x-a)} g_+(\ulw)d\ulw,
	\end{align*}
	where the last equality is the Szeg\"o projection of $H^p(\mathbf R^{m+1}_{+,-a}, \mathbf C^{(m)}).$ The fact $f_+\in H^p(\mathbf R^{m+1}_{+,-a}, \mathbf C^{(m)})$ then follows from the $L^p$-boundedness of the Sezg\"o projection (see e.g. \cite{Gilbert-Murray}). Similarly, one can show $f_-(x)\in H^p(\mathbf R^{m+1}_{-,a}, \mathbf C^{(m)}).$ Thus $f\in H^p(S_a, \mathbf C^{(m)}),1<p<2.$
	
	For $p=1,q=\infty$ we define
	\begin{align*}
	 G_+(x_0+\ulx) = \int_{\mathbf R^m} P_{+,-a}(\ulx-\ulw,x_0-a)g_+(\ulw)d\ulw,
	\end{align*}
	where $P_{+,-a}(\ulx,x_0)=\frac{1}{2\sigma_m}\frac{x_0+2a}{((x_0+2a)^2+|\ulx|^2)^\frac{m+1}{2}}$ is the Poisson kernel on $\mathbf R^{m+1}_{+,-a}.$
	We then have
	\begin{align*}
	f_+(x) &=\frac{1}{(2\pi)^m}\int_{\mathbf R^m} e^{i\langle \ulx,\ulxi\rangle}e^{-(a+x_0)|\ulxi|}\chi_+(\ulxi)\mathcal F(g_+)(\ulxi)d\ulxi\\
	&= \int_{\mathbf R^m} P_{+,-a}(\ulx-\ulw,x_0-a)g_+(\ulw)d\ulw\\
	&=G_+(x).
	\end{align*}
We note that $G_+(x_0+\ulx)\in L^1(\mathbf R^m, \mathbf C^{(m)})$ since
	\begin{align*}
	||G_+(x_0+\ulx)||_{L^1(\mathbf R^m, \mathbf C^{(m)})} &\leq C||g_+||_{L^1(\mathbf R^m, \mathbf C^{(m)})}||P_{+,-a}(\cdot,x_0-a)||_{L^1(\mathbf R^m, \mathbf C^{(m)})}\\
	&=C||g_+||_{L^1(\mathbf R^m, \mathbf C^{(m)})}<\infty,
	\end{align*}
	where $C$ is a positive constant.
	
	The order of taking derivative and taking integral may be exchanged, due to use of the Lebesgue dominated convergence theorem, and thus $f_+$ is left-monogenic on $\mathbf R^{m+1}_{+,-a}$. Thus $f_+\in H^1(\mathbf R^{m+1}_{+,-a}, \mathbf C^{(m)}).$ Similarly, we also have $f_-\in H^1(\mathbf R^{m+1}_{-,a}, \mathbf C^{(m)})$. Therefore, $f\in H^1(S_a, \mathbf C^{(m)}).$ 	
	The uniqueness of (\ref{Decomposition-g}), in fact, is given by Lemma \ref{lem3.5}.
	
	Next we will prove the necessity condition of $f\in H^p(S_a, \mathbf C^{(m)})$ in the theorem.
	Assume that $f\in H^p(S_a, \mathbf C^{(m)}),1\leq p<2.$ The proof of this part is similar to that of the proof of Theorem \ref{PW-strips}. As in Theorem \ref{PW-strips}, we define
	\begin{align*}
	g_\epsilon(x_0+\ulx) = (f(x_0+\cdot)* \phi_\epsilon)(\ulx),
	\end{align*}
	and have $\mathcal F(g_\epsilon(x_0+\cdot))=\mathcal F(f(x_0+\cdot))\mathcal F(\phi_\epsilon),$ which means that supp$\mathcal F(g_\epsilon)$ is compact. By Young's inequality, we have
	\begin{align*}
	\sup_{|x_0|<a}||g_\epsilon(x_0+\cdot)||_{L^p(\mathbf R^m, \mathbf C^{(m)})}\leq C \sup_{|x_0|<a} ||f(x_0+\cdot)||_{L^p(\mathbf R^m, \mathbf C^{(m)})}||\phi_\epsilon||_{L^1(\mathbf R^m, \mathbf C^{(m)})}<\infty,
	\end{align*}
	which amounts that $g_\epsilon \in H^p(S_a, \mathbf C^{(m)}).$ Define
	\begin{align}\label{tem-g}
	G_\epsilon(x)=\frac{1}{(2\pi)^m}\int_{\mathbf R^m}e(x,\ulxi)\mathcal F(g_\epsilon)(\ulxi)d\ulxi, \quad x\in S_a.
	\end{align}
	By the argument used in Theorem \ref{PW-strips}, we have
	\begin{align}\label{eq_tem3}
	\mathcal F(f(x_0+\cdot))(\ulxi)=(e^{-x_0|\ulxi|}\chi_+(\ulxi)+e^{x_0|\ulxi|}\chi_-(\ulxi))\mathcal F(f)(\ulxi).
	\end{align}

	Then, by Hausdorff-Young's inequality, from (\ref{eq_tem3}) we can show that $(e^{a|\ulxi|}-e^{-a|\ulxi|})\mathcal F(f)(\ulxi)\in L^q(\mathbf R^m, \mathbf C^{(m)})$ and
	$(e^{a|\ulxi|}+e^{-a|\ulxi|})\mathcal F(f)(\ulxi)\in L^q(\mathbf R^m, \mathbf C^{(m)})$, which give $e^{a|\ulxi|}\mathcal F(f)(\ulxi)\in L^q(\mathbf R^m, \mathbf C^{(m)}).$
	Finally, as in the proof of Theorem \ref{PW-strips}, applying the Lebesgue dominated convergence theorem to (\ref{tem-g}), we have
	\begin{align*}
	f(x) = \frac{1}{(2\pi)^m}\int_{\mathbf R^m}e(x,\ulxi)\mathcal F(f)(\ulxi)d\ulxi.
	\end{align*}

\end{proof}
\begin{lem}\label{lem3.5}
For $a,b>0,1\leq p<\infty,$ $H^p(\mathbf R^{m+1}_{+,-a}, \mathbf C^{(m)})\cap H^p(\mathbf R^{m+1}_{-,b}, \mathbf C^{(m)})=\{0\}.$
\end{lem}
\noindent\begin{proof}For $f\in H^p(\mathbf R^{m+1}_{+,-a}, \mathbf C^{(m)})
\cap H^p(\mathbf R^{m+1}_{-,b}, \mathbf C^{(m)}),$ using the subharmonicity of $|f|^p,$
 we can show that $|f(x)|$ is bounded, and
 $\lim_{|x_0|\to \infty}|f(x_0+\ulx)|=0.$ Then by Liouville's
 theorem for monogenic functions (see \cite{Brackx-Delanghe-Sommen}), $f(x)$ has to be a constant, and then $f(x)=0.$\end{proof}
\end{remark}

\subsection{$A^2(S_a,\mathbf C^{(m)})$ as a RKHS}
In this section we study the Paley-Wiener theorem of $A^2(S_a, \mathbf C^{(m)}).$
The technique used in the following proof is adapted from \cite{BG}
 (see also \cite{Dang-Mai-Qian}).
\begin{thm}\label{PW-Bergman}
$f\in A^2(S_a, \mathbf C^{(m)})$ if and only if
\begin{align*}
f(x)=\frac{1}{(2\pi)^m}\int_{\mathbf R^m}e(x,\ulxi)g(\ulxi)d\ulxi, \quad x\in S_a,
\end{align*}
where $g(\ulxi)=\mathcal F(f)(\ulxi)\in L^2(\mathbf R^m, \mathbf C^{(m)})$ such that
\begin{align*}
\frac{1}{(2\pi)^m}\int_{\mathbf R^m} (e^{2a|\ulxi|}-e^{-2a|\ulxi|})\frac{|g(\ulxi)|^2}{2|\ulxi|}d\ulxi <\infty.
\end{align*}
Moreover, the Bergman kernel of $A^2(S_a, \mathbf C^{(m)})$ is
\begin{align}\label{Bergman-kernel-Sa}
B(w,\overline x)=\frac{1}{(2\pi)^m}\int_{\mathbf R^m}e(w+\overline x,\ulxi) \frac{2|\ulxi|}{e^{2a|\ulxi|}-e^{-2a|\ulxi|}} d\ulxi.
\end{align}
\end{thm}
\begin{proof}
 For a fixed $0<\delta<a$ let $f^\delta(x)$ be the restriction of $f\in A^2(S_a, \mathbf C^{(m)})$
to $\{x\in S_a:|x_0|<a-\delta\}$. By the subharmonicity of $|f|^2$, we have
\begin{align}\label{subhar1}
\begin{split}
|f(x_0+\ulx)|^2 &\leq \frac{1}{V_\delta}\int_{|y-x|<\frac{\delta}{2}} |f(y_0+\uly)|^2 d\uly dy_0 \\
& \leq \frac{1}{V_\delta}\int_{|y_0|<a-\frac{\delta}{2}}\int_{|\uly-\ulx|<\frac{\delta}{2}}|f(y_0+\uly)|^2  d\uly dy_0\\
& \leq \frac{1}{V_\delta}\int_{-a}^a\int_{|\uly-\ulx|<\frac{\delta}{2}}|f(y_0+\uly)|^2 d\uly dy_0 ,
\end{split}
\end{align}
where $V_\delta=C\delta^{m+1}$ is the volume of the ball $\{y\in \mathbf R^m; |y-x|<\frac{\delta}{2}\}.$
Then, by Fubini's theorem, we have
\begin{align}\label{subhar2}
\begin{split}
\int_{\mathbf R^m}|f(x_0+\ulx)|^2 d\ulx &\leq \frac{1}{V_\delta}\int_{-a}^a \int_{\mathbf R^m}\chi_{|\uly-\ulx|<\frac{\delta}{2}}(\ulx)d\ulx |f(y_0+\uly)|^2 dy_0 d\uly\\
&\leq \frac{C^\prime}{\delta} \int_{-a}^a \int_{\mathbf R^m} |f(y_0+\uly)|^2  d\uly dy_0.
\end{split}
\end{align}
Thus $f^\delta(x)\in H^2(S_{(a-\delta)}, \mathbf C^{(m)}).$ Hence, by Theorem \ref{PW-strips},
there exists $g_\delta$ such that
\begin{align*}
f^\delta(x)=\frac{1}{(2\pi)^m}\int_{\mathbf R^m} e(x,\ulxi)
 g_\delta(\ulxi) d\ulxi,
\end{align*}
where $g_\delta(\ulxi)=\mathcal F(f)(\ulxi)$ satisfies $e^{(a-\delta)|\ulxi|}g_\delta(\ulxi)\in L^2(\mathbf R^m, \mathbf C^{(m)}).$
For any $|x_0|<a,$ we let $\delta = \frac{a-|x_0|}{2}.$ Then, by the above discussion, we have
\begin{align*}
f(x_0+\ulx)=\frac{1}{(2\pi)^m}\int_{\mathbf R^m} e(x,\ulxi)
\mathcal F(f)(\ulxi) d\ulxi,
\end{align*}
and $g(\ulxi)=\mathcal F(f)(\ulxi).$
Furthermore, by Plancherel's theorem, we have
\begin{align}\label{Bergman-2}
\frac{1}{(2\pi)^m}\int_{\mathbf R^m}
 |(e^{-x_0|\ulxi|}\chi_+(\ulxi)g(\ulxi)
 +e^{x_0|\ulxi|}\chi_-(\ulxi)g(\ulxi)|^2 d\ulxi
  = \int_{\mathbf R^m}|f(x_0+\ulx)|^2 d\ulx.
\end{align}
Then by using Fubini's Theorem, we have
\begin{align*}
\frac{1}{(2\pi)^m}\int_{\mathbf R^m}\int_{-a}^a (e^{-2x_0|\ulxi|}|\chi_+(\ulxi)g(\ulxi)|^2+e^{2x_0|\ulxi|}|\chi_-(\ulxi)g(\ulxi)|^2) dx_0 d\ulxi = \int_{-a}^a\int_{\mathbf R^m} |f(x_0+\ulx)|^2 d\ulx dx_0,
\end{align*}
which gives
\begin{align*}
\frac{1}{(2\pi)^m}\int_{\mathbf R^m}(e^{2a|\ulxi|}-e^{-2a|\ulxi|})\frac{|g(\ulxi)|^2}{2|\ulxi|}d\ulxi =\int_{-a}^a\int_{\mathbf R^m} |f(x_0+\ulx)|^2 d\ulx dx_0<\infty.
\end{align*}
 Conversely, if there holds
 \begin{align*}
 f(x) = \frac{1}{(2\pi)^m}\int_{\mathbf R^m} e(x,\ulxi)g(\ulxi)d\ulxi,\quad x\in S_a,
 \end{align*}
 where $g(\ulxi)=\mathcal F(f)(\ulxi)\in L^2(\mathbf R^m,\mathbf C^{(m)})$ such that \begin{align*}
 \frac{1}{(2\pi)^m}\int_{\mathbf R^m}(e^{2a|\ulxi|}-e^{-2a|\ulxi|})\frac{|g(\ulxi)|^2}{2|\ulxi|}d\ulxi<\infty,
 \end{align*} then we can conclude that $f\in A^2(S_a,\mathbf C^{(m)})$ by the above discussion.

In the following we will show (\ref{Bergman-kernel-Sa}).
First, we show that the point-evaluation functional $T_x$ is a linear bounded functional. In fact, (\ref{subhar1}) implies
\begin{align*}
|T_x(f)|=|f(x)| \leq C_x ||f||_{A^2(S_a,\mathbf C^{(m)})}.
\end{align*}
By the Riesz representation theorem, there exists a reproducing kernel function $B(w,\overline x)\in A^2(S_a,\mathbf C^{(m)})$
as a function with respect to $w.$ Then we have
\begin{align}\label{PW-berg-ineq2}
\begin{split}
f(x)&=\frac{1}{(2\pi)^m}\int_{\mathbf R^m} e(x,\ulxi) \mathcal F(f)(\ulxi)d\ulxi \\
&=\int_{-a}^a \int_{\mathbf R^m} \overline B(w,\overline x)f(w)d\ulw dw_0 \\
&= \frac{1}{(2\pi)^m}\int_{-a}^a \int_{\mathbf R^m}\overline {\mathcal F(B)}(\ulxi,\overline x)  (e^{-2w_0|\ulxi|}\chi_+(\ulxi)+e^{2w_0|\ulxi|}\chi_{-}(\ulxi)){\mathcal F(f)}(\ulxi) d\ulxi dw_0,
\end{split}
\end{align}
where we have used (\ref{Plancherel}) and the fact that $\mathcal F(f(w_0+\cdot))(\ulxi)=(e^{-w_0|\ulxi|}\chi_+(\ulxi)+e^{w_0|\ulxi|}\chi_-(\ulxi))\mathcal F(f)(\ulxi),$ and $\mathcal F(f)(\ulxi)$ is the Fourier transform of the restriction of $f$ to $\mathbf R^m.$

 Applying (\ref{PW-berg-ineq2}) to $\widetilde B(x,\overline y)$ (see \textbf{Remark 3} for its definition) and using the uniqueness of the Fourier transform, we can show that
\begin{align}\label{bergman-remark}
\int_{-a}^a\overline{\mathcal F(B)}(\ulxi,\overline x)  (e^{-2w_0|\ulxi|}\chi_+(\ulxi)+e^{2w_0|\ulxi|}\chi_-(\ulxi))dw_0= e(x,\ulxi).
\end{align}
Thus
\begin{align*}
\mathcal F(B)(\ulxi,\overline x)= \frac{2|\ulxi|}{e^{2a|\ulxi|}-e^{-2a|\ulxi|}} e(\overline x,\ulxi),
\end{align*}
and then
\begin{align*}
B(w,\overline x)=\frac{1}{(2\pi)^m}\int_{\mathbf R^m}e(w+\overline x,\ulxi) \frac{2|\ulxi|}{e^{2a|\ulxi|}-e^{-2a|\ulxi|}} d\ulxi.
\end{align*}

\end{proof}

\begin{remark} {\bf 3}\label{3} Combining the arguments used in Theorem \ref{PW-Bergman} and in Lemma \ref{PW-Upper-Hardy},we can prove
\begin{thm}\label{PW-Upper-Bergman}
	$f\in A^2(\mathbf R^{m+1}_{\pm,\mp a},\mathbf C^{(m)})$ if and only if
	\begin{align*}
	f(x) = \frac{1}{(2\pi)^m}\int_{\mathbf R^m} e^{\pm}(x,\ulxi)g(\ulxi)d\ulxi, \quad x\in \mathbf R^{m+1}_{\pm,\mp a},
	\end{align*}
	where $g(\ulxi)=\mathcal F(f)(\ulxi)\in L^2(\mathbf R^m, \mathbf C^{(m)})$ such that
	\begin{align*}
	\frac{1}{(2\pi)^m}\int_{\mathbf R^m} e^{2a|\ulxi|} \frac{|\chi_\pm (\ulxi)g(\ulxi)|^2}{2|\ulxi|}d\ulxi<\infty.
	\end{align*}
\end{thm}

By Theorem \ref{PW-Bergman} and Theorem \ref{PW-Upper-Bergman}, we have
\begin{align*}
A^2(\mathbf R^{m+1}_{+,- a},\mathbf C^{(m)}) \oplus A^2(\mathbf R^{m+1}_{-, a},\mathbf C^{(m)})\subset A^2(S_a,\mathbf C^{(m)}).
\end{align*}

Moreover, the Bergman kernels of $A^2(\mathbf R^{m+1}_{\pm,\mp a},\mathbf C^{(m)})$ are, respectively, given by
\begin{align*}
B_{\pm,\mp a}(w,\overline x) &= \frac{1}{(2\pi)^m}\int_{\mathbf R^m}2|\ulxi|e^{-2a|\ulxi|}e^{\pm}(w+\overline x,\ulxi)d\ulxi\\
&=2\frac{\partial}{\partial x_0} \frac{1}{(2\pi)^m}\int_{\mathbf R^m} e^{-2a|\ulxi|}e^\pm(w+\overline x,\ulxi)d\ulxi\\
&=\mp 2\frac{\partial}{\partial x_0} S_{\pm,\mp a}(w,\overline x),
\end{align*}
where $S_{+,-a}(w,\overline x)$ and $S_{-,a}(w,\overline x)$ are the Szeg\"o kernels given in the previous section. Then we can define $A^2(S_a,\mathbf C^{(m)})\ni \widetilde B(w,\overline x)= B_{+,-a}(w,\overline x) + B_{-,a}(w,\overline x).$


\end{remark}

\begin{remark} {\bf 4}\label{4} Unlike the Hardy space case, we can only give a necessary condition for functions in $A^p(S_a,\mathbf C^{(m)}),1\leq p<2$.
\begin{thm}
If $f\in A^p(S_a,\mathbf C^{(m)}),1\leq p\leq 2,$ then there exists a function $g$ such that
\begin{align*}
f(x)=\frac{1}{(2\pi)^m}\int_{\mathbf R^m} e(x,\ulxi)g(\ulxi)d\ulxi,
\end{align*}
where $g(\ulxi)=\mathcal F(f)(\ulxi)\in L^q(\mathbf R^m,\mathbf C^{(m)}),$ satisfies for $1<p\leq 2,q=\frac{p}{p-1},$
\begin{align}\label{PW-berg-nec1}
\left(\frac{1}{2}\int_{\mathbf R^m} (e^{pa|\ulxi|}-e^{-pa|\ulxi|})^\frac{q}{p}\frac{|\chi_+(\ulxi)g(\ulxi)|^q+|\chi_-(\ulxi)g(\ulxi)|^q}{(p|\ulxi|)^\frac{q}{p}}d\ulxi\right)^\frac{1}{q}\leq C_p||f||_{A^p(S_a,\mathbf C^{(m)})};
\end{align}
and for $p=1,q=\infty,$
\begin{align}\label{PW-berg-nec2}
\frac{1}{2}\sup_{\ulxi\in \mathbf R^m}(e^{a|\ulxi|}-e^{-a|\ulxi|})\frac{|\chi_+(\ulxi)g(\ulxi)|+|\chi_-(\ulxi)g(\ulxi)|}{|\ulxi|}\leq C_1||f||_{A^1(S_a,\mathbf C^{(m)})}.
\end{align}
\end{thm}
\begin{proof}
The proof is similar to that of Theorem \ref{PW-Bergman}. For $1\leq p<2,$ $|f|^p$ is subharmonic, that makes the argument in the proof of Theorem \ref{PW-Bergman} applicable to the present case. In fact, we let $f\in A^p(S_a,\mathbf C^{(m)}),$ and $f^\delta $ the restriction of $f$ to $\{x\in S_a: |x_0|<a-\delta\}.$ By the subharmonicity of $|f|^p$ and the argument used Theorem \ref{PW-Bergman}, we have $f^\delta\in H^p(S_{(a-\delta)},\mathbf C^{(m)}).$
Then, by Theorem \ref{PW-strips-g}, we have
\begin{align*}
f^\delta(x) &= \frac{1}{(2\pi)^m}\int_{\mathbf R^m} e(x,\ulxi) \mathcal F(f)(\ulxi)d\ulxi= \frac{1}{(2\pi)^m}\int_{\mathbf R^m} e(x,\ulxi) g(\ulxi)d\ulxi.
\end{align*}
and
\begin{align*}
\mathcal F(f^\delta(x_0+\cdot))(\ulxi) = (e^{-x_0|\ulxi|}\chi_+(\ulxi)+e^{x_0|\ulxi|}\chi_-(\ulxi))\mathcal F(f)(\ulxi)=(e^{-x_0|\ulxi|}\chi_+(\ulxi)+e^{x_0|\ulxi|}\chi_-(\ulxi))g(\ulxi).
\end{align*}
Since the above equalities holds for all $0<\delta<a,$ we have
\begin{align*}
f(x) &= \frac{1}{(2\pi)^m}\int_{\mathbf R^m} e(x,\ulxi) \mathcal F(f)(\ulxi)d\ulxi= \frac{1}{(2\pi)^m}\int_{\mathbf R^m} e(x,\ulxi) g(\ulxi)d\ulxi.
\end{align*}
and
\begin{align*}
\mathcal F(f(x_0+\cdot))(\ulxi) = (e^{-x_0|\ulxi|}\chi_+(\ulxi)+e^{x_0|\ulxi|}\chi_-(\ulxi))\mathcal F(f)(\ulxi)=(e^{-x_0|\ulxi|}\chi_+(\ulxi)+e^{x_0|\ulxi|}\chi_-(\ulxi))g(\ulxi).
\end{align*}


Next we will prove (\ref{PW-berg-nec1}) and (\ref{PW-berg-nec2}).
We first consider the case $1<p<2,q=\frac{p}{p-1}.$
By Hausdorff-Young's inequality, there holds
\begin{align*}
\left(\int_{\mathbf R^m} |e^{-x_0|\ulxi|}\chi_+(\ulxi)g(\ulxi)+e^{x_0|\ulxi|}\chi_-(\ulxi)g(\ulxi)|^q d\ulxi \right)^\frac{1}{q} \leq C_p\left(\int_{\mathbf R^m}|f(x_0+\ulx)|^p d\ulx\right)^\frac{1}{p}.
\end{align*}

Consequently, using the fact $\chi_+\chi_-=\chi_-\chi_+=0,$ we have
\begin{align*}
\left(\int_{\mathbf R^m} |e^{-x_0|\ulxi|}\chi_+(\ulxi)g(\ulxi)|^q d\ulxi\right)^\frac{1}{q}\leq C_p\left(\int_{\mathbf R^m}|f(x_0+\ulx)|^p d\ulx\right)^\frac{1}{p}
\end{align*}
and
\begin{align*}
\left(\int_{\mathbf R^m} |e^{x_0|\ulxi|}\chi_-(\ulxi)g(\ulxi)|^q d\ulxi\right)^\frac{1}{q}\leq C_p \left(\int_{\mathbf R^m}|f(x_0+\ulx)|^p d\ulx\right)^\frac{1}{p}.
\end{align*}
Then by Minkowski's inequality,
\begin{align*}
\left(\int_{\mathbf R^m} \left(\int_{-a}^a|\chi_+(\ulxi)g(\ulxi)|^pe^{-px_0|\ulxi|}dx_0\right)^\frac{q}{p}d\ulxi\right)^\frac{p}{q} &\leq \int_{-a}^a\left(\int_{\mathbf R^m}|\chi_+(\ulxi)g(\ulxi)|^qe^{-qx_0|\ulxi|}d\ulxi\right)^\frac{p}{q} dx_0 \\
&\leq C_p^p\int_{-a}^a \int_{\mathbf R^m} |f(x_0+\ulx)|^p d\ulx dx_0,
\end{align*}
which gives
\begin{align*}
\left(\int_{\mathbf R^m} (e^{pa|\ulxi|}-e^{-pa|\ulxi|})^\frac{q}{p}
\frac{|\chi_+(\ulxi)g(\ulxi)|^q}{(p|\ulxi|)^\frac{q}{p}}d\ulxi \right)^\frac{1}{q}\leq C_p||f||_{A^p(S_a,\mathbf C^{(m)})}.
\end{align*}
Similarly, there holds
\begin{align*}
\left(\int_{\mathbf R^m} (e^{pa|\ulxi|}-e^{-pa|\ulxi|})^\frac{q}{p}\frac{|\chi_-(\ulxi)g(\ulxi)|^q}{(p|\ulxi|)^\frac{q}{p}}d\ulxi \right)^\frac{1}{q}\leq C_p||f||_{A^p(S_a,\mathbf C^{(m)})}.
\end{align*}
Therefore, for $1<p<2,q=\frac{p}{p-1}$
\begin{align*}
\left(\frac{1}{2}\int_{\mathbf R^m} (e^{pa|\ulxi|}-e^{-pa|\ulxi|})^\frac{q}{p}\frac{|\chi_+(\ulxi)g(\ulxi)|^q+|\chi_-(\ulxi)g(\ulxi)|^q}{(p|\ulxi|)^\frac{q}{p}}d\ulxi\right)^\frac{1}{q}\leq C_p||f||_{A^p(S_a,\mathbf C^{(m)})}.
\end{align*}
For the case $p=1,q=\infty,$ by the definition of the Fourier transform, there holds
\begin{align*}
\sup_{\ulxi\in\mathbf R^m}|e^{-x_0|\ulxi|}\chi_+(\ulxi)g(\ulxi)+e^{x_0|\ulxi|}\chi_-(\ulxi)g(\ulxi)| \leq C_1\int_{\mathbf R^m} |f(x_0+\ulx)| d\ulx,
\end{align*}
which gives
\begin{align*}
\sup_{\ulxi\in \mathbf R^m} |e^{-x_0|\ulxi|}\chi_+(\ulxi)g(\ulxi)|\leq C_1 \int_{\mathbf R^m}|f(x_0+\ulx)|d\ulx
\end{align*}
and
\begin{align*}
\sup_{\ulxi\in \mathbf R^m} |e^{x_0|\ulxi|}\chi_-(\ulxi)g(\ulxi)|\leq C_1 \int_{\mathbf R^m}|f(x_0+\ulx)|d\ulx.
\end{align*}
Consequently, taking integration to the both sides with respect to $x_0$, we have
\begin{align*}
\frac{1}{2}\sup_{\ulxi\in \mathbf R^m}(e^{a|\ulxi|}-e^{-a|\ulxi|})\frac{|\chi_+(\ulxi)g(\ulxi)|+
|\chi_-(\ulxi)g(\ulxi)|}{|\ulxi|}&\leq \int_{-a}^a \int_{\mathbf R^m}|f(x_0+\ulx)|d\ulx dx_0\\
&=C_1 ||f||_{A^1(S_a,\mathbf C^{(m)})}.
\end{align*}

\end{proof}
\end{remark}

Next we give some pointwise estimates of the Bergman kenel $B(w,\overline x).$
\begin{lem}\label{bergman-estimate}
For the Bergman kernel $B(w,\overline x),$ we have
\begin{align}\label{bergman-ineq1}
\frac{c}{(a-|x_0|)^{m+1}}\leq B(x,\overline x)\leq \frac{C}{(a-|x_0|)^{m+1}}
\end{align}
and for $|\ulw+\overline \ulx|\geq 1,$
\begin{align}\label{bergman-ineq22}
|B(w,\overline x)|\leq \frac{M}{|\ulw+\overline \ulx|^\frac{m-1}{2}(2a-|w_0+x_0|)^\frac{m+3}{2}}
\end{align}
for $m=2l+1,l=1,2,...,$ and
\begin{align}\label{bergman-ineq23}
|B(w,\overline x)|\leq \frac{M}{|\ulw+\overline \ulx|^\frac{m-1}{2}(2a-|w_0+x_0|)^\frac{m+4}{2}}
\end{align}
for $m=2l,l=1,2,...,$ where $M$ is a constant that is independent of $w$ and $x.$
\end{lem}
\begin{proof}
We first consider (\ref{bergman-ineq1}). Note that
\begin{align*}
0<||B(\cdot,\overline x)||_{A^2(S_a,\mathbf C^{(m)})}^2 &=B(x,\overline x)\\
&=\frac{1}{(2\pi)^m}\int_{\mathbf R^m}e^{-2x_0|\ulxi|}\chi_+(\ulxi)\frac{2|\ulxi|}{e^{2a|\ulxi|}-e^{-2a|\ulxi|}} d\ulxi \\
& + \frac{1}{(2\pi)^m}\int_{\mathbf R^m}e^{2x_0|\ulxi|}\chi_-(\ulxi)\frac{2|\ulxi|}{e^{2a|\ulxi|}-e^{-2a|\ulxi|}} d\ulxi.
\end{align*}
Thus we must have
\begin{align}\label{berg-ineq3}
\begin{split}
B(x,\overline x)&=\frac{2}{(2\pi)^m}\left(\int_{\mathbf R^m} e^{-2x_0|\ulxi|} \chi_+(\ulxi)\frac{|\ulxi|}{e^{2a|\ulxi|}-e^{-2a|\ulxi|}} d\ulxi +\int_{\mathbf R^m}e^{2x_0|\ulxi|}\chi_-(\ulxi)\frac{|\ulxi|}{e^{2a|\ulxi|}-e^{-2a|\ulxi|}} d\ulxi\right)\\
&= \frac{2}{(2\pi)^m}\int_{S^{m-1}}\left(\int_{0}^\infty e^{-2(x_0+a)r} \frac{(1+\ulxi^\prime)r^m}{1-e^{-4ar}} drd\sigma(\ulxi^\prime) + \int_{0}^\infty e^{2(x_0-a)r} \frac{(1-\ulxi^\prime)r^m}{1-e^{-4ar}} dr \right)d\sigma(\ulxi^\prime)\\
&= \frac{2}{(2\pi)^m}\int_{S^{m-1}}d\sigma(\ulxi^\prime)\int_{0}^\infty \left( e^{-2(x_0+a)r} +e^{2(x_0-a)r}\right) \frac{r^m}{1-e^{-4ar}} dr\\
&\geq \frac{4\sigma_{m-1}}{(2\pi)^m}\int_0^\infty e^{-2(a-|x_0|)r}r^m dr\\
&= \frac{4m!\sigma_{m-1}}{(2\pi)^m} \frac{1}{(2(a-|x_0|))^{m+1}}\\
&\geq \frac{c}{(a-|x_0|)^{m+1}},
\end{split}
\end{align}
where we have used the fact $\int_{S^{m-1}}\ulxi^\prime d\sigma(\ulxi^\prime)=0.$

For the left-hand side of (\ref{bergman-ineq1}), by the third equality of (\ref{berg-ineq3}), we have
\begin{align*}
B(x,\overline x) &\leq \frac{8\sigma_{m-1}}{(2\pi)^m}\left(\int_0^\infty e^{-2(a-|x_0|)r}r^m\sum_{k=0}^\infty e^{-4akr}\right)\\
&=\frac{8\sigma_{m-1}}{(2\pi)^m}\sum_{k=0}^\infty\left(\int_0^\infty e^{-2(a-|x_0|)r}r^m e^{-4akr}\right)\\
&=\frac{8m!\sigma_{m-1}}{(2\pi)^m2^{m+1}}\sum_{k=0}^\infty \frac{1}{((2k+1)a-|x_0|)^{m+1}}\\
&\leq \frac{8m!\sigma_{m-1}}{(2\pi)^m2^{m+1}}\sum_{k=0}^\infty \frac{1}{(2ka+a-|x_0|)^{m+1}}\\
&\leq \frac{1}{(a-|x_0|)^{m+1}}\frac{4m!C_m}{(2\pi)^m2^{m+1}} \sum_{k=0}^\infty\frac{1}{(\frac{2ka}{a-|x_0|}+1)^{m+1}}\\
&\leq \frac{1}{(a-|x_0|)^{m+1}}\frac{4m!C_m}{(2\pi)^m2^{m+1}} \sum_{k=0}^\infty\frac{1}{(2k+1)^{m+1}}\\
&=\frac{C}{(a-|x_0|)^{m+1}}.
\end{align*}

In the following we will prove (\ref{bergman-ineq22}) and (\ref{bergman-ineq23}) by using the argument similar to Lemma \ref{PW-infty}. First we recall that
$\omega_\frac{m-2}{2}=\Gamma(\frac{m-1}{2})\Gamma(\frac{1}{2}),$ and $J_k(t)$ is the Bessel function given by $$J_k(t)=\frac{(\frac{t}{2})^k}{\omega_k}\int_{-1}^1 e^{its}(1-s^2)^{\frac{2k-1}{2}}ds,\quad k>-\frac{1}{2}.$$

Note that
\begin{align*}
B(w,\overline x)&=\frac{1}{(2\pi)^m}\int_{\mathbf R^m}\frac{|\ulxi|}{e^{2a|\ulxi|}-e^{-2a|\ulxi|}}\left(e^{-(w_0+x_0)|\ulxi|}(1+i\frac{\ulxi}{|\ulxi|})+e^{(w_0+x_0)|\ulxi|}(1-i\frac{\ulxi}{|\ulxi|})\right)e^{i\langle \ulw+\overline \ulx,\ulxi\rangle}d\ulxi\\
&=\frac{1}{(2\pi)^m}\int_{\mathbf R^m}\frac{|\ulxi|}{e^{2a|\ulxi|}-e^{-2a|\ulxi|}}\left(e^{-(w_0+x_0)|\ulxi|}+e^{(w_0+x_0)|\ulxi|}\right)e^{i\langle \ulw+\overline \ulx,\ulxi\rangle}d\ulxi\\
&+\frac{1}{(2\pi)^m}\int_{\mathbf R^m}\frac{|\ulxi|}{e^{2a|\ulxi|}-e^{-2a|\ulxi|}}\left(e^{-(w_0+x_0)|\ulxi|}-e^{(w_0+x_0)|\ulxi|}\right)i\frac{\ulxi}{|\ulxi|}e^{i\langle \ulw+\overline \ulx,\ulxi\rangle}d\ulxi\\
&=I_1+I_2.
\end{align*}

As in Lemma \ref{PW-infty}, we have
\begin{align*}
|I_1|&=\left|\frac{1}{(2\pi)^m}\int_0^\infty \int_{S^{m-1}}\frac{\left(e^{-(w_0+x_0)r}+e^{(w_0+x_0)r}\right)r^m}{e^{2ar}-e^{-2ar}}e^{i\langle \ulw+\overline{\ulx},\ulxi^\prime \rangle}d\sigma(\ulxi^\prime)dr\right|\\
&=\left|\frac{\omega_\frac{m-2}{2}}{(2\pi)^m}\int_0^\infty \frac{\left(e^{-(w_0+x_0)r}+e^{(w_0+x_0)r}\right)r^m}{e^{2ar}-e^{-2ar}}(r|\ulw+\overline{\ulx}|)^{-\frac{m-2}{2}}J_\frac{m-2}{2}(r|\ulw+\overline{\ulx}|)dr\right|\\
&\leq \frac{M_1}{|\ulw+\overline{\ulx}|^\frac{m-1}{2}}\int_0^\infty \frac{\left(e^{-(2a+w_0+x_0)r}+e^{-(2a-(w_0+x_0))r}\right)r^\frac{m+1}{2}}{1-e^{-4ar}}dr\\
&\leq \frac{2M_1}{|\ulw+\overline{\ulx}|^\frac{m-1}{2}}\int_0^\infty \frac{e^{-(2a-|w_0+x_0|)r}r^\frac{m+1}{2}}{1-e^{-4ar}}dr\\
&=\frac{2M_1}{|\ulw+\overline{\ulx}|^\frac{m-1}{2}}\sum_{k=0}^\infty\int_0^\infty e^{-(2a(2k+1)-|w_0+x_0|)r}r^\frac{m+1}{2} dr.
\end{align*}
When $m=2l+1,l\geq 1,$
\begin{align*}
|I_1|&\leq \frac{2M_1}{|\ulw+\overline \ulx|^\frac{m-1}{2}}\sum_{k=0}^\infty\frac{1}{(2a(2k+1)-|w_0+x_0|)^\frac{m+3}{2}}\\
&\leq \frac{2M_1}{|\ulw+\overline{\ulx}|^\frac{m-1}{2}(2a-|w_0+x_0|)^\frac{m+3}{2}}\sum_{k=0}^\infty\frac{1}{(2k+1)^\frac{m+3}{2}}\\
&\leq \frac{M_2}{|\ulw+\overline{\ulx}|^\frac{m-1}{2}(2a-|w_0+x_0|)^\frac{m+3}{2}}.
\end{align*}
When $m=2l,l\geq 1,$
\begin{align*}
|I_1|&\leq \frac{2M_1}{|\ulw+\overline{\ulx}|^\frac{m-1}{2}}\sum_{k=0}^\infty\frac{1}{(2a(2k+1)-|w_0+x_0|)^\frac{m}{2}}\int_0^\infty e^{-(2a(2k+1)-|w_0+x_0|)r}r^\frac{1}{2}dr\\
&\leq \frac{2M_1}{|\ulw+\overline{\ulx}|^\frac{m-1}{2}}\sum_{k=0}^\infty\left(\frac{1}{(2a(2k+1)-|w_0+x_0|)^\frac{m+2}{2}}+\frac{1}{(2a(2k+1)-|w_0+x_0|)^\frac{m+4}{2}}\right)\\
&\leq \frac{2M_1}{|\ulw+\overline{\ulx}|^\frac{m-1}{2}(2a-|w_0+x_0|)^\frac{m+4}{2}}\sum_{k=0}^\infty\left(\frac{2a-|w_0+x_0|}{(2k+1)^\frac{m+2}{2}}+\frac{1}{(2k+1)^\frac{m+4}{2}}\right)\\
&\leq \frac{M_2^\prime}{|\ulw+\overline{\ulx}|^\frac{m-1}{2}(2a-|w_0+x_0|)^\frac{m+4}{2}}.
\end{align*}
For $I_2$, we note that
\begin{align*}
I_2 &=\frac{i}{(2\pi)^m}\int_0^\infty \int_{S^{m-1}}\frac{\left(e^{-(w_0+x_0)r}+e^{(w_0+x_0)r}\right)r^m}{e^{2ar}-e^{-2ar}}e^{ir\langle \ulw+\overline\ulx,\ulxi^\prime\rangle} \ulxi^\prime d\sigma(\ulxi^\prime) dr.
\end{align*}
Using exactly the same argument as in Lemma \ref{PW-infty}, we introduce $I_{21}$ and $I_{22}$, where
\begin{align*}
I_{21}=\int_0^\pi e^{ir|\ulw+\overline \ulx|\cos\theta_1} \cos\theta_1 (\sin \theta_1)^{m-2}d\theta_1=\frac{i\omega_\frac{m}{2}}{m-1} (r|\ulw+\overline\ulx|)^{-\frac{m}{2}+1}J_{\frac{m}{2}}(r|\ulw+\overline\ulx|)
\end{align*}
and
\begin{align*}
I_{22}=\int_0^\pi e^{ir|\ulw+\overline\ulx|\cos\theta_1} \sin\theta_1 (\sin \theta_1)^{m-2}d\theta_1=\omega_\frac{m-1}{2}(r|\ulx|)^{-\frac{m-1}{2}}J_{\frac{m-1}{2}}(r|\ulw+\overline\ulx|).
\end{align*}
As in Lemma \ref{PW-infty} again, to estimate $I_2,$ it suffices to estimate
\begin{align}\label{berg-ineq4}
\left| \int_0^\infty\frac{\left(e^{-(w_0+x_0)r}+e^{(w_0+x_0)r}\right)r^m}{e^{2ar}-e^{-2ar}}(r|\ulw+\overline\ulx|)^{-\frac{m}{2}+1}J_{\frac{m}{2}}(r|\ulw+\overline\ulx|) dr\right|
\end{align}
and
\begin{align}\label{berg-ineq5}
\left| \int_0^\infty\frac{\left(e^{-(w_0+x_0)r}+e^{(w_0+x_0)r}\right)r^m}{e^{2ar}-e^{-2ar}}(r|\ulw+\overline\ulx|)^{-\frac{m-1}{2}}J_{\frac{m-1}{2}}(r|\ulw+\overline\ulx|) dr\right|.
\end{align}
For (\ref{berg-ineq4}), we have
\begin{align*}
(\ref{berg-ineq4})
&=\left|\sum_{k=0}^\infty\frac{1}{|\ulw+\overline \ulx|^\frac{m}{2}}\int_0^\infty (e^{-(2a(2k+1)+w_0+x_0)r}+e^{-(2a(2k+1)-(w_0+x_0))})r^{\frac{m}{2}+1}|\ulw+\overline\ulx|J_\frac{m}{2}(r|\ulw+\overline \ulx|)dr\right|\\
&\leq \frac{M_3}{|\ulw+\overline \ulx|^\frac{m-1}{2}}\sum_{k=0}^\infty\int_0^\infty\left| (e^{-(2a(2k+1)+w_0+x_0)r}+e^{-(2a(2k+1)-(w_0+x_0))})r^\frac{m+1}{2}\right|dr\\
&\leq \frac{2M_3}{|\ulw+\overline \ulx|^\frac{m-1}{2}}\sum_{k=0}^\infty\int_0^\infty e^{-(2a(2k+1)-|w_0+x_0|)r}r^\frac{m+1}{2}dr.
\end{align*}
Similar to the discussion for $I_1,$ we have that
\begin{align*}
(\ref{berg-ineq4})
&\leq \frac{M_4}{|\ulw+\overline \ulx|^\frac{m-1}{2}(2a-|w_0+x_0|)^\frac{m+3}{2}}
\end{align*}
for $m=2l+1,l=1,2,...,$
and
\begin{align*}
(\ref{berg-ineq4})
&\leq \frac{M_4^\prime}{|\ulw+\overline{\ulx}|^\frac{m-1}{2}(2a-|w_0+x_0|)^\frac{m+4}{2}}
\end{align*}
for $m=2l,l=1,2,...$.\\
\noindent For (\ref{berg-ineq5}), we have
\begin{align*}
(\ref{berg-ineq5})&=
\left| \sum_{k=0}^\infty\frac{1}{|\ulw+\overline \ulx|^\frac{m-1}{2}}\int_0^\infty\left(e^{-(2a(2k+1)+w_0+x_0)r}+e^{-(2a(2k+1)-(w_0+x_0))r}\right)r^{\frac{m+1}{2}}J_{\frac{m-1}{2}}(r|\ulw+\overline\ulx|) dr\right|\\
&\leq \frac{M_5}{|\ulw+\overline \ulx|^\frac{m}{2}}\sum_{k=0}^\infty\int_0^\infty e^{-(2a(2k+1)-|w_0+x_0|)r}r^\frac{m}{2}dr.
\end{align*}
Consequently,
\begin{align*}
(\ref{berg-ineq5})\leq \frac{M_6}{|\ulw+\overline \ulx|^\frac{m}{2}(2a-|w_0+x_0|)^\frac{m+3}{2}}
\end{align*}
for $m=2l+1,l=1,2,...,$
and
\begin{align*}
(\ref{berg-ineq5})\leq \frac{M_6^\prime}{|\ulw+\overline \ulx|^\frac{m}{2}(2a-|w_0+x_0|)^\frac{m+2}{2}}
\end{align*}
for $m=2l,l=1,2,....$

Therefore, we have for $|\ulw+\overline \ulx|\geq 1,$
\begin{align*}
|B(w,\overline x)|\leq \frac{M}{|\ulw+\overline \ulx|^\frac{m-1}{2}(2a-|w_0+x_0|)^\frac{m+3}{2}}
\end{align*}
for $m=2l+1,l=1,2,...,$ and
\begin{align*}
|B(w,\overline x)|\leq \frac{M}{|\ulw+\overline \ulx|^\frac{m-1}{2}(2a-|w_0+x_0|)^\frac{m+4}{2}}
\end{align*}
for $m=2l,l=1,2,...,$ where $M$ is a constant that is independent of $w$ and $x.$

\end{proof}

%

\end{document}